\documentclass[11pt,a4paper,twoside]{amsart}
	
\usepackage[T1]{fontenc}
\usepackage[utf8]{inputenc}
\usepackage[english,ngerman]{babel}	
\usepackage{geometry}

\usepackage{hyperref}
\usepackage{graphicx}
\usepackage{tabularx}

\usepackage{amsmath}
\usepackage{amssymb}			
\usepackage{amsfonts}
\usepackage{amsthm}
\usepackage{mathtools}
\usepackage{mathrsfs}
\usepackage{stmaryrd}
\usepackage{enumitem}
\usepackage[ruled]{algorithm2e}

\usepackage{tikz}
  \usetikzlibrary{matrix}
  \usetikzlibrary{fit}
  \usetikzlibrary{backgrounds}
  \usetikzlibrary{arrows}
  \usetikzlibrary{shapes}

\theoremstyle{plain}
\newtheorem{theo}{Theorem}[section]
\newtheorem{lemm}[theo]{Lemma}
\newtheorem{prop}[theo]{Proposition} 
\newtheorem{cor}[theo]{Corollary}
\newtheorem{con}[theo]{Conjecture}
\newtheorem*{cons}{Conjecture}

\theoremstyle{definition}
\newtheorem{defi}[theo]{Definition}

\newtheorem{rem}[theo]{Remark}

\def\QQ{{\mathbb Q}}
\def\CC{{\mathbb C}}

\def\PP{{\mathbb P}}

\def\NN{{\mathbb N}}
\def\ZZ{{\mathbb Z}}

\def\AA{{\mathbb A}}
\def\DD{{\mathbb D}}

\def\Ta{{\mathbb T}}
\def\Qp{{\QQ_p}}
\def\Hp{{\mathcal{H}_p}}
\def\GL{{\mathrm{GL}_2(\QQ_p)}}
\def\SL{{\mathrm{SL}_2(\QQ_p)}}
\def\TT{{\mathcal{T}}}

\def\simeq{\cong}
\def\epsilon{{\varepsilon}}
\def\phi{{\varphi}}

\begin{document}
\selectlanguage{english}
\title[$p$-adic automorphic forms and Teitelbaum's $\mathcal{L}$-invariant]{A control theorem for $p$-adic automorphic forms and Teitelbaum's $\mathcal{L}$-invariant}
\author{Peter Mathias Gräf}
\date{\today}
\subjclass[2010]{Primary: 11F03, 11F85; Secondary: 11F67, 20E08}
\address{IWR, University of Heidelberg, Im Neuenheimer Feld 205, 69120 Heidelberg, Germany}
\email{peter.graef@iwr.uni-heidelberg.de}

\maketitle

\begin{abstract}
In this article, we describe an efficient method for computing Teitelbaum's $p$-adic $\mathcal{L}$-invariant. These invariants are realized as the eigenvalues of the $\mathcal{L}$-operator acting on a space of harmonic cocycles on the Bruhat-Tits tree $\TT$, which is computable by the methods of Franc and Masdeu described in \cite{fm}. The main difficulty in computing the $\mathcal{L}$-operator is the efficient computation of the $p$-adic Coleman integrals in its definition. To solve this problem, we use overconvergent methods, first developed by Darmon, Greenberg, Pollack and Stevens. In order to make these methods applicable to our setting, we prove a control theorem for $p$-adic automorphic forms of arbitrary even weight. Moreover, we give computational evidence for relations between slopes of $\mathcal{L}$-invariants of different levels and weights for $p=2$.
\end{abstract}


\section{Introduction}
Let $f$ be a newform of even weight $k\geq 2$ for $\Gamma_0(pN)$, where $p$ is prime and $N$ is an integer not divisible by $p$. Let $\chi$ be a Dirichlet character of conductor prime to $pN$ with $\chi(p)= 1$. By the work of Mazur and Swinnerton-Dyer, see \cite{mtt}, there exists a $p$-adic $L$-function $L_p(f,\chi,s)$ attached to $f$ that interpolates the algebraic parts $L^{\text{alg}}(f,\chi,j)$ for $j\in\{1,\dots,k-1\}$ of the special values of the classical $L$-function attached to $f$. If $f$ is an eigenform for the $U_p$-operator with eigenvalue $p^{\frac{k}{2}-1}$, Mazur, Tate and Teitelbaum showed in \cite{mtt} that the order of vanishing of the $p$-adic $L$-function attached to $f$ at the central point $s=\frac{k}{2}$ is one higher than that of the classical $L$-function attached to it. In connection with a $p$-adic formulation of BSD-type conjectures, they formulated the following conjecture.
\begin{cons}[Exceptional zero conjecture]
There exists an invariant $\mathcal{L}_p(f)\in\CC_p$, depending only on the local Galois representation $\sigma_p(f)$ attached to $f$, such that
\[
L'_p(f,\chi,\tfrac{k}{2})=\mathcal{L}_p(f)\cdot L^{\mathrm{alg}}(f,\chi,\tfrac{k}{2}).
\]
\end{cons}
After Mazur, Tate and Teitelbaum gave a candidate for $\mathcal{L}_p(f)$  when $f$ corresponds to an elliptic curve in \cite{mtt}, several possible definitions for the invariant $\mathcal{L}_p(f)$ in the general case have been proposed. The first candidate for $\mathcal{L}_p(f)$, denoted by $\mathcal{L}_{\text{Tei}}(f)$, was introduced by Teitelbaum in \cite{tei90} and relies on the Jacquet-Langlands correspondence and the $p$-adic uniformization theory of \v{C}erednik and Drinfeld. A second invariant $\mathcal{L}_{\text{Col}}(f)$ was proposed by Coleman in \cite{col}, based on his theory of $p$-adic integration on modular curves. A third invariant $\mathcal{L}_{\text{Fon}}(f)$ was suggested by Fontaine and Mazur in \cite{maz} using Fontaine's classification of $p$-adic representations. A fourth invariant $\mathcal{L}_{\text{Ort}}(f)$ was defined by Darmon (in weight two) and Orton (in the general case) in \cite{dar01} and \cite{ort} using certain modular form-valued distributions. Finally, a fifth invariant $\mathcal{L}_{\text{Bre}}(f)$ building on $p$-adic Langlands theory is due to Breuil, see \cite{br}. It is known that all of these invariants are equal when they are defined and the exceptional zero conjecture has been proved for all of them, see for example \cite{dar01}, \cite{ort}, \cite{em} and \cite{bdi}. A more detailed account of the various proofs can be found in \cite{col10}.  A fundamental observation in proving the exceptional zero conjecture first due to Greenberg and Stevens, \cite{gs}, is the relation between $\mathcal{L}_p(f)$ and the (essentially unique) $p$-adic family of eigenforms passing through $f$. If we denote this family by 
\[
f_\kappa=\sum_{n=1}^\infty a_n(\kappa)q^n \quad \text{with} \quad f_k=f
\]
where the coefficients $a_n(\kappa)$ are rigid analytic on a disc containing $k$ in the weight space $\mathcal{W}=\mathrm{Hom}_{\mathrm{cont}}(\ZZ_p^\times,\QQ^\times_p)$, this relation is given as
\[
\mathcal{L}_p(f)=-2\operatorname{\mathrm{dlog}}(a_p(\kappa))|_{\kappa=k}.
\]
The qualitative behaviour of $\mathcal{L}_p(f)$ is well known, on the contrary, not a lot of quantitative data on $\mathcal{L}_p(f)$ for arbitrary even weight is available.\par
\vspace{\baselineskip}
In this article, we describe an efficient method for the computation of $\mathcal{L}_p(f)$. For this purpose, we focus on the invariant $\mathcal{L}_{\text{Tei}}(f)$ defined by Teitelbaum. Despite its disadvantages (mainly that it is only defined when $f$ can be associated to a modular form on a Shimura curve via the Jacquet-Langlands correspondence), we highly benefit from its accessibility for explicit computations.\par

By the theorems of Jacquet-Langlands \cite{jac} and \v{C}erednik-Drinfeld \cite{dri}, Teitelbaum's $\mathcal{L}$-invariant is defined in terms of rigid analytic modular forms defined on the $p$-adic upper half plane $\Hp$ with respect to a group $\Gamma$ arising from a definite quaternion algebra defined over $\QQ$. These rigid analytic modular forms admit a nice combinatorial description as $\Gamma$-invariant harmonic cocycles on the Bruhat-Tits tree $\TT$ for $\GL$. Teitelbaum's invariants $\mathcal{L}_p(f)$ are then realized as the eigenvalues of an operator, called the $\mathcal{L}$-operator, defined on the finite-dimensional $\CC_{p}$-vector space of harmonic cocycles. In this article, we describe an efficient method to compute the $\mathcal{L}$-operator (up to a prescribed precision) as a matrix with $p$-adic entries, whose eigenvalues can then be analyzed using standard techniques such as Newton polygons. There are three main difficulties one has to overcome:
\begin{itemize}
\item[(1)] The $\mathcal{L}$-operator in \cite{tei90} is defined over $\CC_p$. From the definitions, it is straightfoward to see that it can in fact be defined over the quadratic unramified extension of $\Qp$. However, in order to keep the running time of our computations as low as possible, we show that the $\mathcal{L}$-operator can in fact be defined over $\QQ_p$ by slightly modifiying one of the two maps entering into its definition.
\item[(2)] In order to describe $\Gamma$-invariant harmonic cocycles on $\TT$ by a finite amount of data, it is necessary to compute a fundamental domain for the action of $\Gamma$ on $\TT$. In \cite{fm}, Franc and Masdeu described and implemented an algorithm to compute such a fundamental domain, which is the basis of our computations.
\item[(3)] There are certain Coleman integrals entering into the definition of the $\mathcal{L}$-operator following \cite{tei90}, whose efficient computation seems to be completely out of reach. Teitelbaum himself proved in \cite{tei90} that one can replace these Coleman integrals by $p$-adic integrals directly coming from harmonic cocycles, and he was able to compute the invariant $\mathcal{L}_p(f)$ for a modular form of weight $4$ by using a process of polynomial approximation and Riemann integration. However, this method is much too slow to compute the $\mathcal{L}$-operator efficiently in the general case. Thankfully, an alternate approach was presented by Greenberg in \cite{grea} building on the overconvergent methods developed by Darmon, Pollack and Stevens in \cite{dp} and \cite{ps} and extended to our setting by Franc and Masdeu in \cite{fm}. Roughly speaking, the data needed for the computation is encoded in values of certain rigid analytic automorphic forms, whose computation is a much simpler task.
\end{itemize}

In our method, we apply the overconvergent methods mentioned above in the specific setting of computing the $\mathcal{L}$-operator. To make these methods applicable, we proof a control theorem (Theorem \ref{stevens}) for $p$-adic automorphic forms of arbitrary even weight. This theorem is a generalization of \cite[Corollary 2]{grea} and stated in \cite[Section 6.1]{fm}. In Remark \ref{gencontr} we explain the full generality in which our theorem holds.\par
This article is arranged as follows. In Section \ref{lop}, we introduce the Bruhat-Tits tree as a skeleton of the $p$-adic upper half plane $\mathcal{H}_p$ and define Teitelbaum's $\mathcal{L}$-operator in terms of $p$-adic integration on $\mathcal{H}_p$. In Section \ref{autf} we introduce $p$-adic automorphic forms and prove the control theorem. In Section \ref{compu}, we describe how to apply the methods of Section \ref{autf} to compute the $\mathcal{L}$-operator building on \cite{fm}. In Section \ref{exa}, we conclude with some examples and give computational evidence for relations between slopes of $\mathcal{L}$-invariants of different levels and weights for $p=2$.

All computations have been done using a \texttt{Sagemath} implementation (see \cite{sage}) of the algorithms. The code builds heavily on the algorithms of Franc and Masdeu and is available on request.  This article is based on the authors Masters's thesis \cite{mine}.\par

The author wishes to thank Gebhard Böckle for suggesting this interesting topic and for his support and encouragement. He is grateful to Marc Masdeu, Tommaso Centeleghe and Samuele Anni for many enlightening discussions and their comments on this article. The author was supported by the DFG via the Forschergruppe 1920 and the SPP 1489. Moreover, part of this work was done during a stay at Concordia University, Montr\'eal funded by a DAAD-Doktorandenstipendium and the author expresses his gratitude to Adrian Iovita and the Mathematics department of Concordia University for their hospitality.\par
This is a pre-print of an article published in the Ramanujan Journal. The final authenticated version is available online at: \url{http://dx.doi.org/10.1007/s11139-019-00160-1}\par

\section{Harmonic cocyles and modular forms}
\label{lop}

In this section, we introduce the Bruhat-Tits tree $\TT$ and the spaces of harmonic cocycles, and describe Teitelbaum's Poisson kernel. Then we develop all the background we need to define the $\mathcal{L}$-operator.

\subsection{The Bruhat-Tits tree}
\label{btt}
\begin{defi}
The \emph{Bruhat-Tits tree} $\mathcal{T}$ for $\text{GL}_2(\QQ_p)$ is the graph whose vertices are the homothety classes of $\ZZ_p$-lattices in $\QQ_p^2$. Two vertices $v$ and $w$ are joined by an edge if there exist representative lattices $L$ and $L'$ such that
\[
pL\subsetneq L'\subsetneq L.
\]
For a given lattice $L$ in $\QQ_p^2$ we denote the homothety class by $[L]$. It is well known that $\mathcal{T}$ is a $p+1$-regular tree. The \emph{set of vertices} of $\mathcal{T}$ is denoted by $\mathcal{T}_0$ and the \emph{set of edges} by $\mathcal{T}_1$. For an edge $e\in\mathcal{T}_1$ we denote by $s(e)\in\mathcal{T}_0$ the \emph{source} and by $t(e)\in\mathcal{T}_0$ the \emph{target}. The edge $\overline{e}\in\mathcal{T}_1$ with $s(\overline{e})=t(e)$ and $t(\overline{e})=s(e)$ is called the \emph{opposite edge}. Furthermore we denote the \emph{distance} between two vertices $v,w\in\mathcal{T}_0$ by $d(v,w)$.
\end{defi}

The group $\GL$ acts transitively on $\TT$ as follows. For a vertex $[L]\in\TT_0$ and $g\in\GL$ we set $g[L]=[g L]$, where the action on the lattice is given by matrix multiplication with respect to the standard coordinates. We set $v_0=[\ZZ_p^2]$. The vertex $v_1=\left(\begin{smallmatrix} 0&1\\ p&0\end{smallmatrix}\right) v_0$ is adjacent to $v_0$. Let $e_0$ denote the edge with $s(e_0)=v_1$ and $t(e_0)=v_0$. We have 
\[
\operatorname{\mathrm{Stab}}_\GL(v_0)=\QQ^\times_p\textnormal{GL}_2(\ZZ_p)\quad \text{and}\quad
\operatorname{\mathrm{Stab}}_\GL(e_0)=\QQ^\times_p\Gamma_0(p\ZZ_p),
\]
where $\Gamma_0(p\ZZ_p)=\left\{\left(\begin{smallmatrix} a &b\\ c&d\end{smallmatrix}\right)\in\mathrm{GL}_2(\ZZ_p) \ \middle| \ p\mid c\right\}$. As a consequence, we obtain natural bijections
$\GL/\QQ^\times_p\textnormal{GL}_2(\ZZ_p)\simeq \TT_0$ and $\GL/\QQ^\times_p\Gamma_0(p\ZZ_p)\simeq \TT_1$.

The tree $\TT$ can be viewed as a skeleton of the rigid analytic $p$-adic upper half plane $\Hp=\PP^1(\CC_p)\setminus \PP^1(\QQ_p)$. A detailed description of its properties can be found in \cite[Section 1]{dt}. The group $\GL$ acts on $\mathcal{H}_p$ via fractional linear transformations.  The following theorem gives a precise description of the relation between the two objects. Let $\mathcal{X}_0$ denote the following standard affinoid,
\[
\mathcal{X}_0=\{z\in\PP^1(\CC_p) \mid |z|_p\leq 1, |z-j|_p\geq 1 \text{ for } j=0,\dots,p-1\}\subset\Hp
\]
where $|\cdot|_p$ denotes the $p$-adic absolute value, normalized such that $|p|_p=1$.
We also set $W_0=\{z\in\PP^1(\CC_p)\mid 1<|z|_p<p\}\subset\Hp$ and denote by $\TT_\QQ$ the geometric realization of $\TT$ over $\QQ$.

\begin{prop}[{\cite[Section 1.3]{dt}}]
\label{redmap}
There is a continuous and  surjective $\GL$-equivariant map
\[
\operatorname{\mathrm{red}}\colon\Hp\rightarrow\TT_\QQ
\]
that for $g\in\GL$ maps the affinoids $g\mathcal{X}_0$ to the vertices $gv_0$ of $\TT_\QQ$ and maps the annuli $gW_0$ to the edges $ge_0$ of $\TT_\QQ$. \end{prop}

For an arbitrary edge $e\in\TT_1$ we denote the open annulus $\mathrm{red}^{-1}(e)$ by $W_e$. Let $\mathcal{B}$ denote the set of compact open balls in $\PP^1(\Qp)$. The group $\GL$ acts on $\mathcal{B}$ via fractional linear transformations. For an edge $e$ we denote by $\TT_e$ the maximal subtree of $\TT$ containing $e$ and not containing any other edge $e'$ with $s(e')=s(e)$ and by $V_e$ the closure of $\operatorname{\mathrm{red}}^{-1}(\TT_e)$ in $\PP^1(\CC_p)$. Then we obtain the following (see \cite[Section 5.2, Proof of Theorem 5.9]{dar}).

\begin{prop}
\label{ballbij}
The assignment $e\mapsto U_e=V_e\cap\PP^1(\Qp)$ gives rise to a $\GL$-equivariant bijection $\TT_1\rightarrow\mathcal{B}$ mapping $e_0$ to $\ZZ_p$. 
\end{prop}

Consequently, we have 
\[U_{\overline{e}}=\PP^1(\Qp)\setminus U_{e} \quad \text{for } e\in\TT_1 \quad \text{and}\quad
\PP^1(\Qp)=\bigsqcup_{s(e)=v}U_e \quad \text{for } v\in\TT_0.
\]

Now we consider group actions on the tree $\TT$. We are only interested in groups arising from rational definite quaternion algebras as follows.  let $B$ be a definite quaternion algebra over $\QQ$ of discriminant $N^{-}$, coprime to $p$. Let $N^+$ be a positive integer coprime to $pN^{-}$ and $R$ be an Eichler $\ZZ$-order of level $N^{+}$ in B. Then $R[\frac{1}{p}]$ is an Eichler $\ZZ[\frac{1}{p}]$-order. Let $\Gamma^{(p)}_{N^{+},N^{-}}=R[\frac{1}{p}]^\times_1$ denote the group of elements of reduced norm $1$. We fix a splitting $\iota\colon B_p\rightarrow\mathrm{M}_2(\QQ_p)$ such that $\iota(R_p)=\mathrm{M}_2(\ZZ_p)$, where $R_p=R\otimes_\ZZ\ZZ_p$. We may regard $\Gamma^{(p)}_{N^{+},N^{-}}$ as a subgroup of $\SL$ via the splitting and we obtain an action on $\TT$. In the sequel, we drop the dependencies on $p$, $N^-$ and $N^{+}$ and write $\Gamma=\Gamma^{(p)}_{N^+,N^-}$. We collect some important results on the action of $\Gamma$ on $\mathcal{T}$.

\begin{prop}[{\cite[Section I.3]{gvdp} and \cite[Section II.1, Theorem 5]{se}}]~
\label{actprop}
\begin{itemize}
\item[(a)] The group $\Gamma\subset\SL$ is finitely generated, discrete and cocompact and acts without inversion on $\TT$. The quotient $\Gamma\textbackslash\mathcal{T}$ is a finite graph.
\item[(b)] The stabilizers $\operatorname{\mathrm{Stab}}_\Gamma(v)$ and $\operatorname{\mathrm{Stab}}_\Gamma(e)$ are finite for all $v\in\TT_0$, $e\in\TT_1$.
\end{itemize}
\end{prop}

\subsection{Harmonic cocycles}
\label{coc}
Let $V$ be a finite dimensional $\Qp$-vector space with a left action of $\Gamma$.
We denote by $C(\mathcal{T},V)$ the $\Qp$-vector space of $V$-valued functions on the edges of the Bruhat-Tits tree satisfying $c(\overline{e})=-c(e)$ for all $e\in\TT_1$. There is a left action of $\Gamma$ on $C(\mathcal{T},V)$ given by
\[
(\gamma\cdot c)(e)=\gamma\cdot(c(\gamma^{-1}e)),
\]
where $c\in C(\mathcal{T},V)$, $e\in\mathcal{T}_1$, $\gamma\in\Gamma$. We set $C(\Gamma,V)=C(\mathcal{T},V)^\Gamma$.
\begin{defi}
A function $c\in C(\mathcal{T},V)$ is called a \emph{harmonic cocycle} on $\mathcal{T}$ if
\begin{align*}
\sum_{s(e)=v}c(e)=0 \quad \text{for all } v\in\mathcal{T}_0.
\end{align*}
The space of harmonic cocycles is denoted by $C_h(\mathcal{T},V)$. It is easily verified that this space is $\Gamma$-stable and thus we may set $C_h(\Gamma,V)=C_h(\mathcal{T},V)^\Gamma$.
\end{defi}
The $\Qp$-vector space we are mainly interested in is given as follows. For an even integer $k\geq0$, we denote by $\mathcal{P}_k\subset\QQ_p[x]$ the finite dimensional $\Qp$-vector space of polynomials of degree at most $k$. We define a right action of $\text{GL}_2(\QQ_p)$ on $\mathcal{P}_k$ as follows.
\[
(P\cdot_k g)(x)=\det(g)^{-\frac{k}{2}}(cx+d)^kP\left(\frac{ax+b}{cx+d}\right), \quad \text{for } P\in\mathcal{P}_k,\ g=\begin{pmatrix} a&b\\c&d \end{pmatrix}\in\text{GL}_2(\QQ_p).
\]
Let $V_k=\text{Hom}_{\QQ_p}(\mathcal{P}_k,\QQ_p)$ denote the linear dual of $\mathcal{P}_k$ with the induced left action of $\text{GL}_2(\QQ_p)$. Since $V_k$ is finite dimensional, the space $C_h(\Gamma, V_k)$ is finite dimensional by Proposition \ref{actprop}.

\subsection{Rigid analytic modular forms}
\label{poi}
For an even integer $k\geq 0$ let $A_k^\mathrm{loc}$ denote the $\Qp$-vector space of $\Qp$-valued functions on $\PP_1(\Qp)$ which are locally analytic, except possibly for a pole at $\infty$ of order at most $k$. Let $\mathcal{P}_k^\mathrm{loc}\subset A_k^\mathrm{loc}$ be the subspace of locally polynomial functions of degree at most $k$ in one variable over $\Qp$. We endow $A_k^\mathrm{loc}$ with the Fr\'echet topology, \cite[Definition 2.1.4]{dt}. The group $\GL$ acts on $A_k^\mathrm{loc}$ as follows.
\[
(f\cdot_k g)(x)=\det(g)^{-\frac{k}{2}}(cx+d)^kf\left(\frac{ax+d}{cx+d}\right), \quad \text{for }  f\in A_k^\mathrm{loc},\  g=\begin{pmatrix}a&b\\ c&d\end{pmatrix}\in\GL.
\]
There is an induced action on $\mathcal{P}_k^\mathrm{loc}$ which is compatible with the inclusion $\mathcal{P}_k\subset\mathcal{P}_{k}^\mathrm{loc}$. For a more detailed description of these spaces see \cite[Section 2.1]{dt}. Furthermore, let $\mathcal{K}=\{U\subset\PP^1(\Qp)\ | \  U \text{ compact open}\}$.

\begin{defi}~
\begin{itemize}
\item[(a)] A \emph{measure on $\mathcal{P}_k^\mathrm{loc}$} is an element $\mu\in\text{Hom}_{\Qp}(\mathcal{P}_k^\mathrm{loc},\Qp)$. For $f\in\mathcal{P}_k^\mathrm{loc}$, $U\in\mathcal{K}$ we write
\[
\int_Uf(x)\text{d}\mu(x)=\mu(f\chi_U).
\]
\item[(b)] A \emph{measure on $A_k^\mathrm{loc}$} is an element $\mu\in\text{Hom}_{\text{cont}}(A_k^\mathrm{loc},\Qp)$. For $f\in A_k^\mathrm{loc}$, $U\in\mathcal{K}$ we write
\[
\int_Uf(x)\text{d}\mu(x)=\mu(f\chi_U).
\]
\end{itemize}
\end{defi}

In particular, when integrating a convergent Taylor series against such a measure, the sum and integral may be interchanged. We want to attach to a harmonic cocycle $c\in C_h(\Gamma,V_k)$ a measure $\mu_c$ on $\mathcal{P}_k^\mathrm{loc}$ and we will describe how to extend this measure uniquely to a measure on $A_k^\mathrm{loc}$.  We define this measure by
\[
\int_{U_e}P(x)\text{d}\mu_c(x)=c(e)(P)\in \Qp 
\]
for $P \in\mathcal{P}_k$. Note that this completely determines the measure on $\mathcal{P}_k^\mathrm{loc}$. The properties of $c$ have the following consequences.

\begin{lemm}
\label{zero}
The measure defined by $c\in C_h(\Gamma,V_k)$ has the following properties:
\begin{itemize}
\item[(a)] $\int_{\PP^1(\QQ_p)}P(x)\textnormal{d}\mu_c(x)=0$ for all $P\in\mathcal{P}_k$.
\item[(b)] $\int_{\gamma U} f(x)\textnormal{d}\mu_c(x)=\int_{U} (f\cdot_k\gamma)(x)\textnormal{d}\mu_c(x)$ for $\gamma\in\Gamma$, $f\in\mathcal{P}_k^\mathrm{loc}$, $U\in\mathcal{K}$.
\end{itemize}
\end{lemm}

By a theorem of Amice-Velu and Vishik, we can integrate functions in $A_k^\mathrm{loc}$ agaist $\mu_c$.

\begin{theo}[Amice-Velu, Vishik, {\cite[Proposition 9]{tei90}}]
\label{amice}
Let $c\in C_h(\Gamma,V_k)$. There is a unique extension of $\mu_c$ to a measure on $A_k^\mathrm{loc}$ characterized by the following properties:
\begin{itemize}
\item[(a)] $\int_{U_e}P(x)\text{d}\mu_c(x)=c(e)(P)$ for all $P\in\mathcal{P}_k$, $e\in\mathcal{T}_1$.
\item[(b)] There exists a constant $C\geq 0$ such that for all $e\in\mathcal{T}_1$ with $\infty\in U_e$, $0\notin U_e$ and $n\leq k$ we have
\[
\left|\int_{U_e} x^n\textnormal{d}\mu_c(x)\right|_p\leq C\rho^{-n+k/2},
\]
and for $a\in U_e\subset\Qp$ and $n\geq 0$ we have
\[
\left|\int_{U_e} (x-a)^n\textnormal{d}\mu_c(x)\right|_p\leq C\rho^{n-k/2}.
\]
Here, $\rho=\sup_{z\in U_e}|1/u|_p$ if $\infty\in U_e$ and $\rho=\sup_{z,z'\in U_e}|z-z'|_p$ if $\infty\notin U_e$.

\end{itemize}

\end{theo}

By \cite[Lemma 10]{tei90}, property (b) of Lemma \ref{zero} extends to all functions $f \in A_k^\mathrm{loc}$. For $c\in C_h(\Gamma,V_k)\otimes_{\Qp}\CC_p$ we define a function $f_c:\mathcal{H}_p\rightarrow \CC_p$ by
\[
f_c(\tau)=\int_{\PP^1(\Qp)}\frac{1}{\tau-x}\text{d}\mu_c(x) \quad \text{for } \tau\in\mathcal{H}_p.
\]
Then $f_c$ is rigid analytic and satifies the following modular transformation property:
\[
f_c(\gamma\tau)=(c\tau+d)^{k+2}f_c(\tau), \quad \text{for }  \tau\in\mathcal{H}_p,\ \gamma=\begin{pmatrix} a&b\\c&d \end{pmatrix}\in\Gamma.
\]
The $\CC_p$-vector space $\mathcal{S}^{\textnormal{rig}}_{k+2}(\Gamma,\CC_p)$ of \emph{rigid analytic modular forms} is the space of all rigid analytic functions on $\mathcal{H}_p$ satisfying this transformation property. Both of these spaces carry actions of the Hecke algebra attached to the group $\Gamma$ by the usual double coset decomposition. A detailed description of this Hecke algebra can be found in \cite{koc}.
\begin{theo}[Drinfeld, Manin, Schneider, {\cite[Theorem 3]{tei90}}]
\label{poiiso}
The map
\begin{align*}
\phi\colon C_h(\Gamma,V_k)\otimes_{\Qp}\CC_p&\rightarrow\mathcal{S}^{\textnormal{rig}}_{k+2}(\Gamma,\CC_p)\\
c&\mapsto f_c
\end{align*}
is a Hecke-equivariant isomorphism for all  even integers $k\geq 0$. The inverse is given by $f\mapsto c_f$, where $c_f(e)(P)=\operatorname{\mathrm{res}}_{e}(P(\tau)f(\tau)\text{d}\tau)$.
\end{theo}

\subsection{The $\mathcal{L}$-operator}
\label{mapsi}
For $c\in C_h(\Gamma,V_k)$ and $v\in\mathcal{T}_0$ we define a map $\psi^v(c)\colon\Gamma\rightarrow V_k$ by
\[
\psi^v(c)(\gamma)=\sum_{e\colon v\rightarrow\gamma v} c(e)\in V_k \quad \text{for } \gamma\in\Gamma.
\]

\begin{theo}[de Shalit, Schneider]
\label{psiso}
The map $\psi^v\colon C_h(\Gamma,V_k)\rightarrow Z^1(\Gamma, V_k)$ is a well-defined homomorphism of $\Qp$-vector spaces and the induced map
\[
\psi\colon C_h(\Gamma,V_k)\rightarrow H^1(\Gamma, V_k)
\]
is a Hecke-equivariant isomorphism independent of the choice of $v\in\TT_0$.
\end{theo}
\begin{proof}
The Hecke-equivariance is straightfoward. The remaining statements are proved in \cite[Section 6]{des}.
\end{proof}
\begin{rem}
The action of the Hecke algebra on $H^{1}(\Gamma,V_k)$ is the usual (see \cite[Section 1.1]{as}). The action on $C_h(\Gamma, V_k)$ can be interpreted analogously by the identification $C_h(\Gamma, V_k)=H^0(\Gamma,C_h(\mathcal{T},V_k))$.
\end{rem}

The second map we need in order to define the $\mathcal{L}$-operator builds on the integration theory developed in Subsection \ref{poi}. Let $K_p$ denote the quadratic unramified extension of $\Qp$ and $\mathcal{O}_p$ its ring of integers. We set
\[
\Hp(K_p)=\PP^1(K_p)\setminus\PP^1(\QQ_p)\subset\Hp.
\]

\begin{defi}
Let $c\in C_h(\Gamma,V_k)$ be a harmonic cocycle and denote the associated measure by $\mu_c$. Fix the branch of the $p$-adic logarithm $\log_p\colon K_p^\times\rightarrow \mathcal{O}_p$ such that $\log_p(p)=0$. For $P\in \mathcal{P}_k$ and $\tau_1,\tau_2\in\Hp(K_p)$ we define
\[
\int_{\tau_1}^{\tau_2}\omega_c(P)=\int_{\PP^1(\Qp)}P(x)\log_p\left(\frac{x-\tau_2}{x-\tau_1}\right)\text{d}\mu_c(x)\in K_p.
\]
Note that the integrand is locally analytic with the right pole order at infinity and, therefore, can be integrated against $\mu_c$.
\end{defi}

\begin{rem}
In \cite[Theorem 4]{tei90} it is shown that the above integral in fact coincides with the branch of the $p$-adic Coleman integral of the function $P(\tau)f_{c}(\tau)$ corresponding to our choice of logarithm, which justifies the notation of a line integral.
\end{rem}

For simplicity of notation, we set $\operatorname{\mathrm{Tr}}=\tfrac{1}{2}\operatorname{\mathrm{Tr}}_{K_p/\Qp}\colon K_p\rightarrow\Qp$.
For $c\in C_h(\Gamma,V_k)$ and $\tau\in\Hp(K_p)$ we define a map  $\lambda^\tau(c)\colon\Gamma\rightarrow V_k$ by
\[
\lambda^\tau(c)(\gamma)(P)=\operatorname{\mathrm{Tr}}\left(\int_{\tau}^{\gamma\tau}\omega_c(P)\right)\in\Qp, 
\quad \text{for } \gamma\in\Gamma,\ P\in\mathcal{P}_k.
\]
\begin{prop}[{\cite[Lemma 7]{tei90}}]
The map $\lambda^\tau\colon C_h(\Gamma,V_k)\rightarrow Z^1(\Gamma, V_k)$ is a well-defined homomorphism of $\Qp$-vector spaces. The induced map
\[
\lambda\colon C_h(\Gamma,V_k)\rightarrow H^1(\Gamma, V_k)
\]
is a Hecke-equivariant homomorphism independent of the choice of $\tau\in\Hp(K_p)$.
\end{prop}
Now we can define the $\mathcal{L}$-operator.

\begin{defi}
The Hecke-equivariant homomorphism
\[
\mathcal{L}=\lambda\circ\psi^{-1}\colon H^1(\Gamma,V_k)\rightarrow H^1(\Gamma, V_k)
\]
is called the \emph{$\mathcal{L}$-operator} of weight $k+2$ for the group $\Gamma$ defined over $\Qp$.
\end{defi}

\begin{rem}
\label{LCp}
Note that  the trace $\operatorname{\mathrm{Tr}}$ appearing in our definition is a new feature not present in the original definition due to Teitelbaum in \cite{tei90}. However, by extending scalars to $\CC_p$ we obtain the $\mathcal{L}$-operator defined in \cite{tei90}, since we do not change the cohomology classes over $\CC_p$ (or $K_p$): if we denote the nontrivial element of $\mathrm{Gal}(K_p/\QQ_p)$ by $\sigma$, the cocycle $\lambda^\tau(c)$ is the average of two cocycles corresponding to $\tau$ and $\tau^\sigma$ that differ only by a coboundary defined over $K_p$. We modified the definition for computational reasons. We want to do most of our computations over $\Qp$ and avoid working over $K_p$ whenever possible. Thereby, we have shown that Teitelbaum's $\mathcal{L}$-operator is in fact defined over $\QQ_p$.
\end{rem}

Let $N=N^{+}\cdot N^{-}$. The Jacquet-Langlands correspondence (see \cite{jac}) together with the $p$-adic uniformization theorem of \v{C}erednik and Drinfeld (see \cite{dri}) imply that there is a Hecke-equivariant isomorphism 
\[
\mathcal{S}_{k+2}(\Gamma_0(pN),\CC_p)^{pN^{-}\textnormal{-new}}\simeq\mathcal{S}_{k+2}^{\mathrm{rig}}(\Gamma,\CC_p).
\]
Therefore, we can define the $\mathcal{L}$-invariant attached to a newform $f\in\mathcal{S}_{k+2}(\Gamma_0(pN),\CC_p)^{\mathrm{new}}$. By Theorem \ref{poiiso}, we get a corresponding harmonic cocycle $c_f\in C_h(\Gamma,V_{k})\otimes_\Qp\CC_p$ and by the multiplicity one principle, we find that there is a unique scalar $\mathcal{L}_p(f)\in\CC_p$ such that
\[
\mathcal{L}(\psi(c_f))=\mathcal{L}_p(f)\cdot \psi(c_f).
\]

\begin{defi}
Let $f\in\mathcal{S}_{k+2}(\Gamma_0(pN),\CC_p)^{\mathrm{new}}$ be a newform. The scalar $\mathcal{L}_p(f)\in\CC_p$ constructed above is called \emph{(Teitelbaum's) $\mathcal{L}$-invariant} of $f$.
\end{defi}
The $\mathcal{L}$-operator and the $\mathcal{L}$-invariant attached to a newform $f\in\mathcal{S}_{k+2}(\Gamma_0(pN),\CC_p)^{\mathrm{new}}$ are independent of the choice of the fixed splitting $\iota\colon B_p\rightarrow\mathrm{M}_2(\Qp)$, \cite[Theorem 2]{tei90}.

\section{$p$-adic automorphic forms}
\label{autf}
The main difficulty in computing the $\mathcal{L}$-operator is the efficient computation of the integrals appearing in the definition of $\lambda$. In this section, we describe an efficient way to compute for $c\in C_h(\Gamma,V_k)$, $i\in\NN_0$ and $g\in\GL$, the \emph{i-th $g$-moment of $\mu_c$} 
\[
 m(\mu_c,g,i)=\int_{g\ZZ_p}x^i\cdot_k g^{-1}\text{d}\mu_c(x)\in\QQ_p.
\]
If these moments are known to sufficient precision, we can break up $\PP^1(\QQ_p)$ into compact open balls such that the integrand has a nice analytic expression on each open ball and compute the integrals efficiently. We show that the moments are encoded as values of certain rigid analytic automorphic forms  in Theorem \ref{speclift} and we describe an efficient way to compute these values up to a prescribed precision in Subsection \ref{liftsec}. This is an \emph{overconvergent method} based on the one developed by Darmon, Pollack and Stevens. In \cite{grec}, Greenberg describes a similar method in the case of modular symbols. We adapt his method to our setting and thereby give proofs for the statements in \cite[Section 6]{fm}.

\subsection{Coefficient modules}
Recall that $\Gamma_0(p\ZZ_p)=\left\{\left(\begin{smallmatrix} a &b\\ c&d\end{smallmatrix}\right)\in\mathrm{GL}_2(\ZZ_p) \ \middle| \ p\mid c\right\}$.
\begin{defi}
A \emph{coefficient module} is a $\Qp$-vector space endowed with a right action of $\Gamma_0(p\ZZ_p)$. An \emph{integral coefficient module} is a $\ZZ_p$-module with a right action of $\Gamma_0(p\ZZ_p)$.
\end{defi}

The vector space $V_k$ becomes a coefficient module in the following way. We regard the right $\GL$-module $\mathcal{P}_k$ as a left $\GL$-module by setting 
\[
(g\cdot_k P)(x)= (P\cdot_kg^{-1})(x), \quad \text{for } g\in\GL, P\in\mathcal{P}_k
\]
and endow $V_k$ with the induced right $\GL$-action. Let $\Sigma_0(p)$ denote the following monoid
\[
\Sigma_0(p)=\left\{\begin{pmatrix} a &b\\ c&d\end{pmatrix}\in \mathrm{M}_{2}(\ZZ_p) \biggm| p\mid c, a\in\ZZ_p^\times, ad-bc\neq 0\right\}.
\]
Note that $\Gamma_0(p\ZZ_p)\subset\Sigma_0(p)$. The other coefficient module we are mainly interested in is given as follows. Let $\Ta$ denote the Tate algebra in one variable over $\Qp$,
\[
\Ta=\left\{\sum_{n=0}^\infty a_n x^n\in\Qp\llbracket x\rrbracket \Biggm| \lim_{n\rightarrow\infty} a_n=0\right\}.
\]
The sup-norm equips $\Ta$ with the structure of a $p$-adic Banach space. For an even integer $k\geq 0$ we define a continuous left action of $\QQ^\times_p\Sigma_0(p)$ on $\Ta$. For $\sigma=\left(\begin{smallmatrix} a &b\\ c&d\end{smallmatrix}\right)\in\Sigma_0(p)$, $u\in\QQ^\times_p$ and $f\in\Ta$ set
\[
(u\sigma\cdot_k f)(x)=\det(\sigma)^{-\frac{k}{2}}(-cx+a)^kf\left(\frac{dx-b}{-cx+a}\right).
\]
Note that this is well-defined since $c$ is divisible by $p$, $a$ is a $p$-adic unit and matrices of the form $\sigma=\left(\begin{smallmatrix} u &0\\ 0&u\end{smallmatrix}\right)\in\Sigma_0(p)$, for $u\in\ZZ_p^\times$, act trivially. Let $\Ta_k$ denote the space $\Ta$ with the above weight-$k$ action.
\begin{defi}
The continuous dual $\DD_k=\text{Hom}_{\text{cont}}(\Ta_k, \QQ_p)$ with the induced right action of $\QQ^\times_p\Sigma_0(p)$ is called the space of \emph{rigid analytic distributions} of weight $k$. Since $\Gamma_0(p\ZZ_p)\subset\Sigma_0(p)$, this is a coefficient module.
\end{defi}

\begin{rem}
The natural inclusion $\mathcal{P}_k\rightarrow\Ta_k$ is left $\QQ^\times_p\Sigma_0(p)$-equivariant. By duality, we obtain a surjection $\pi\colon\DD_k\rightarrow V_k$, which is right $\QQ^\times_p\Sigma_0(p)$-equivariant.
\end{rem}

Let
\begin{align*}
&\DD_k(\ZZ_p)=\{\omega\in\DD_k \mid \omega(x^i)\in\ZZ_p \text{ for all } i\in\NN_0\}
\intertext{and}
&V_k(\ZZ_p)=\{\omega\in V_k \mid \omega(x^i)\in\ZZ_p \text{ for all } i\in\{0,\dots,k\}\}.
\end{align*}
It is straightfoward to check that $\DD_k(\ZZ_p)$ and $V_k(\ZZ_p)$ are $\Gamma_0(p\ZZ_p)$-stable and thus define integral coefficient modules. Clearly, the map $\pi$ respects these integral structures.

\begin{lemm}[{\cite[Lemma 1]{grec}}]
\label{bound}
Let $\omega\in\DD_k$. Then the moments $\omega(x^i)$ are uniformly bounded, i.e. $\DD_k\simeq \DD_k(\ZZ_p)\otimes_{\ZZ_p}\Qp$.
\end{lemm}

The space $\DD_k(\ZZ_p)$ admits a filtration:
\begin{align*}
F^0\DD_k(\ZZ_p)&=\{\omega\in\DD_k(\ZZ_p) \mid \omega(x^i)=0 \text{ for all } i\in\{0,\dots, k\}\},\\
F^n\DD_k(\ZZ_p)&=\{\omega\in F^0\DD_k(\ZZ_p) \mid \omega(x^{k+i})\in p^{n-i+1}\ZZ_p, \text{ for all } i\in\{1,\dots, n\}\}, \text{ for } n\geq 1.
\end{align*}

\begin{lemm}[{\cite[Lemma 2]{grec}}]
\label{stable}
The sets $F^n\DD_k(\ZZ_p)$ are $\Gamma_0(p\ZZ_p)$-stable and thus define integral coefficient modules.
\end{lemm}

Thanks to the above lemma, we can introduce another class of coefficient modules.

\begin{defi}
Let $n\geq 0$. The \emph{$n$-th approximation module} to the module $\DD_k(\ZZ_p)$ is the integral coefficient module
\[
A^n\DD_k(\ZZ_p)=\DD_k(\ZZ_p)/F^n\DD_k(\ZZ_p).
\] 
Note that $A^n\DD_k(\ZZ_p)$ is a finitely generated $\ZZ_p$-module and we have $A^0\DD_k(\ZZ_p)\simeq V_k(\ZZ_p)$.
\end{defi}

\begin{prop}[{\cite[Proposition 4.4]{pp}}]
\label{projdist}
The natural projection
\[
\DD_k(\ZZ_p)\rightarrow\varprojlim_{n} A^n\DD_k(\ZZ_p)
\]
is an isomorphism.
\end{prop}

\subsection{$p$-adic automorphic forms}


\begin{defi}
Let $V$ be an (integral) coefficient module. A $\Gamma$-invariant \emph{$p$-adic automorphic form} on $\GL$ with values in $V$ is a left $\QQ^\times_p\Gamma$-invariant and right $\Gamma_0(p\ZZ_p)$-equivariant map $\phi\colon \GL\rightarrow V$. The space of $\Gamma$-invariant $V$-valued $p$-adic automorphic forms on $\GL$ is denoted by $\AA(\Gamma,V)$.
\end{defi}

\begin{rem}
\label{cosrep}
Let $V$ be an (integral) coefficient module. An element $\phi\in\AA(\Gamma, V)$ is completely determined by its values on a finite set  $B_\Gamma=(b_j)_{j\in J}\subset\GL$ of representatives for the double coset space
\[
\QQ^\times_p\Gamma\textbackslash\GL/\Gamma_0(p\ZZ_p),
\]
which corresponds bijectively to the edges of $\Gamma\textbackslash\TT$. Assume that the elements in the coefficient module $V$ can be represented (at least up to a finite $p$-adic precision) by a finite amount of data. Then it is possible to store $\phi$ as a vector $(\phi(b_j))_{j\in J}$ of elements in $V$, which allows one to compute in the space $\AA(\Gamma, V)$.
\end{rem}

\begin{defi}~
\begin{itemize}
\item[(a)] The space of \emph{$p$-adic automorphic forms of weight $k$} for the group $\Gamma$ is the $\Qp$-vector space $\AA_k(\Gamma)=\AA(\Gamma, V_k)$.
\item[(b)] The space of \emph{rigid analytic automorphic forms of weight $k$} for the group $\Gamma$ is the $\Qp$-vector space $\mathcal{A}_k(\Gamma)=\AA(\Gamma, \DD_k)$.
\item[(c)] The $\Gamma_0(p\ZZ_p)$-equivariant surjection $\pi\colon\DD_k\rightarrow V_k$ induces a map $\rho\colon\mathcal{A}_k(\Gamma)\rightarrow\AA_k(\Gamma)$, which is called the \emph{specialization map}.
\end{itemize}
\end{defi}

We also have integral structures on these spaces of automorphic forms by setting
\[
\AA_k(\Gamma,\ZZ_p)=\AA(\Gamma,V_k(\ZZ_p)) \quad \text{and}\quad
\mathcal{A}_k(\Gamma,\ZZ_p)=\AA(\Gamma,\DD_k(\ZZ_p)).
\]
The specialization map induces a well-defined map $\rho\colon\mathcal{A}_k(\Gamma,\ZZ_p)\rightarrow\AA_k(\Gamma,\ZZ_p)$.

\begin{prop}
\label{multi}
We have $\AA_k(\Gamma,\ZZ_p)\otimes_{\ZZ_p}\Qp\simeq\AA_k(\Gamma)$ and $\mathcal{A}_k(\Gamma,\ZZ_p)\otimes_{\ZZ_p}\Qp\simeq\mathcal{A}_k(\Gamma)$.
\end{prop}

\begin{proof}
This is an immediate consequence of the fact that a $p$-adic automorphic form is completely determined by its values on a finite set of coset representatives for
\[
\QQ^\times_p\Gamma\textbackslash\GL/\Gamma_0(p\ZZ_p)
\]
(see Remark \ref{cosrep}) together with the fact that $V_k$ is finite dimensional and  Lemma \ref{bound}.
\end{proof}

Our next aim is to show that the specialization map is surjective and to construct explicit lifts that are easily computable.  For each $b_j$ as in Remark \ref{cosrep} let
\[
S_j=\Gamma_0(p\ZZ_p)\cap b_j^{-1}\Gamma b_j=b_j^{-1}\operatorname{\mathrm{Stab}}_\Gamma(b_je_0)b_j
\]
and note that $S_j$ is finite by Proposition \ref{actprop}. Let $\theta\colon\Ta_k\rightarrow \mathcal{P}_k$ be the natural truncation. It is a left inverse of the natural inclusion $\mathcal{P}_k\rightarrow\Ta_k$. Note that $\theta$ is a surjective homomorphism, but in general not $\Gamma_0(p\ZZ_p)$-equivariant.

\begin{prop}
\label{initiallift}
The specialization map $\rho\colon\mathcal{A}_k(\Gamma)\rightarrow\AA_k(\Gamma)$ is surjective. More specifically, let $\phi\in\AA_k(\Gamma)$. There is a unique lift $\Phi_0\in\mathcal{A}_k(\Gamma)$ of $\phi$ satisfying
\[
\Phi_0(b_j)(f)=\frac{1}{|S_j|}\sum_{\sigma\in S_j}\phi(b_j)(\theta(\sigma\cdot_k f))
\]
for all $f\in\Ta_k$.
\end{prop}

\begin{proof}
Note that $\Phi_0$ is completely determined by the above values together with the left $\QQ^\times_p\Gamma$-invariance and the right $\Gamma_0(p\ZZ_p)$-equivariance. Thus, it remains to show that $\Phi_0$ is well-defined. Note that $\Phi_0(b_j)\in\DD_k$ since $\Gamma_0(p\ZZ_p)$ acts continuously on $\Ta_k$. Now let $\sigma'\in\Gamma_0(p\ZZ_p)$ with $b_j\sigma' =\gamma b_j$ for some $j$ and $\gamma\in\Gamma$, hence $\sigma'\in S_j$. We need to show that $\Phi_0(b_j\sigma')=\Phi_0(b_j)$. For $f\in\Ta_k$ we compute
\begin{align*}
\Phi_0(b_j\sigma')(f)&=\Phi_0(b_j)(\sigma' \cdot_k f)=\frac{1}{|S_j|}\sum_{\sigma\in S_j}\phi(b_j)(\theta(\sigma\sigma'\cdot_k f))\\
&=\frac{1}{|S_j|}\sum_{\sigma\in S_j}\phi(b_j)(\theta(\sigma\cdot_k f))=\Phi_0(b_j)(f).
\end{align*}
For $i\in\{0,\dots, k\}$ we see, by definition of $S_j$, that
\[
\Phi_0(b_j)(x^i)=\frac{1}{|S_j|}\sum_{\sigma\in S_j}\phi(b_j)(\sigma\cdot_k x^i)=\frac{1}{|S_j|}\sum_{\sigma\in S_j}\phi(b_j\sigma)(x^i)=\phi(b_j)(x^i)
\]
and hence, $\Phi_0$ lifts $\phi$.
\end{proof}
 
We introduce some more spaces of automorphic forms, which are very useful in our computations. For $n\geq 0$ let
\[
\mathcal{A}_{k,n}(\Gamma,\ZZ_p)=\mathcal{A}(\Gamma, F^n\DD_k(\ZZ_p)) \quad \text{and} \quad
\mathcal{A}^n_k(\Gamma,\ZZ_p)=\mathcal{A}(\Gamma, A^n\DD_k(\ZZ_p)).
\]
The natural projection $\DD_k(\ZZ_p)\rightarrow A^n\DD_k(\ZZ_p)$ defines a specialization map
\[
\rho^n\colon\mathcal{A}_k(\Gamma,\ZZ_p)\rightarrow\mathcal{A}^n_k(\Gamma,\ZZ_p).
\]
Furthermore, for all $n\geq m\geq 0$ we obtain maps
\[
\rho^n_m\colon\mathcal{A}_k^{n}(\Gamma,\ZZ_p)\rightarrow\mathcal{A}_k^m(\Gamma,\ZZ_p)
\]
such that $\rho^n_m=\rho^l_m\circ \rho^n_l$  and $\rho^m=\rho^{n}_{m}\circ\rho^n$ for all $m\leq l\leq n$. Note that $\rho^0=\rho$ with the identification $A^0\DD_k(\ZZ_p)\simeq V_k(\ZZ_p)$.

\begin{prop}
\label{projlim}
The natural projection
\[
\mathcal{A}_k(\Gamma,\ZZ_p)\rightarrow\varprojlim_{n}\mathcal{A}_k^n(\Gamma,\ZZ_p)
\]
is an isomorphism. 
\end{prop}

\begin{proof}
This is an immediate consequence of Proposition \ref{projdist}.
\end{proof}

We want to introduce the $U_p$-operator on automorphic forms. We set $\alpha_j=\left(\begin{smallmatrix} p &j\\ 0&1\end{smallmatrix}\right)$ for $j=0,\dots, p-1$.
Then we have
\[
\Gamma_0(p\ZZ_p)\alpha_0\Gamma_0(p\ZZ_p)=\bigsqcup_{j=0}^{p-1}\alpha_j\Gamma_0(p\ZZ_p).
\]
Note that $\alpha_j^{-1}\in\QQ^\times_p\Sigma_0(p)$ and thus $\alpha_j^{-1}$ acts on $\DD_k$.
\begin{defi}
The \emph{$U_p$-operator} on $\mathcal{A}_k(\Gamma)$ is the $\Qp$-linear map $U_p\colon\mathcal{A}_k(\Gamma)\rightarrow \mathcal{A}_k(\Gamma)$ given by
\[
(U_p\Phi)(g)=p^{\frac{k}{2}}\sum_{j=0}^{p-1}\Phi(g\alpha_j)\cdot_k\alpha_j^{-1}, \quad \text{for } g\in\GL, \Phi\in\mathcal{A}_k(\Gamma).
\]
Furthermore, we define $\widetilde{U_p}=p^{-\frac{k}{2}}U_p$.
\end{defi}

It is straightforward to check that $U_p$ is well-defined and does not depend on the choice of the $\alpha_j$. More explicitly, we have
\begin{align*}
(U_p\Phi)(g)(x^i)=p^{\frac{k}{2}}\sum_{j=0}^{p-1}\Phi(g\alpha_j)(\alpha_j^{-1}\cdot_k x^i)=\sum_{j=0}^{p-1}\sum_{\nu=0}^i\binom{i}{\nu}p^\nu j^{i-\nu}\Phi(g\alpha_j)(x^\nu).
\end{align*}
This formula shows that $U_p$ preserves the subspaces $\mathcal{A}_k(\Gamma,\ZZ_p)$ and $\mathcal{A}_{k,n}(\Gamma,\ZZ_p)$ and is well-defined on $\mathcal{A}_k^n(\Gamma,\ZZ_p)$ for all $n\geq 0$. Moreover, the specialization maps $\rho^n$ and $\rho^n_m$  and, consequently, the isomorphism in Proposition \ref{projlim} are $U_p$-equivariant.

\subsection{Lifting $U_p$-eigenforms}
\label{liftsec}
We want to prove the following theorem, which is a generalization of \cite[Corollary 2]{grea} to automorphic forms of higher weight and in the case of modular symbols originally due to Stevens (see \cite{ste}). A similar theorem with some minor errors is stated in \cite[Theorem 6.1]{fm}.
\begin{theo} The restriction of the specialization map 
\label{stevens}
\[
\rho\colon\mathcal{A}_k(\Gamma)^{U_p= p^{k/2}}\rightarrow\AA_k(\Gamma)^{U_p= p^{k/2}}
\]
is an isomorphism.
\end{theo}
Our proof below will give an efficient algorithm to compute the unique lift and we will show that we can use this lift to compute the moments attached to a harmonic cocycle. First we need to introduce more notation. Let
\[
V_k'(\ZZ_p)=\{\omega\in V_k(\ZZ_p) \mid \omega(x^i)\in p^{\frac{k}{2}-i}\ZZ_p \text{ for all } i\in\{0,\dots,\tfrac{k}{2}\}\}.
\]
The following lemma is a straighfoward computation.
\begin{lemm}
The set $V_k'(\ZZ_p)$ is $\Gamma_0(p\ZZ_p)$-stable and thus defines an integral coefficient module.
\end{lemm}

The next proposition and its corollary are the reason why we are interested in the coefficient module $V_k'(\ZZ_p)$. We will see that the conditions we impose on the low moments guarantee that the $\widetilde{U_p}$-operator is well-defined on the corresponding space of automorphic forms. Note that in the case $k=0$ we have $V_k'(\ZZ_p)=V_k(\ZZ_p)$, so that no further conditions are imposed.

\begin{prop}[{\cite[Lemma 11]{grec}}]
\label{helps}
 For $j\in\{0,\dots,p-1\}$ let $\alpha_j\in\GL$ be defined as in the previous section.
\begin{itemize}
\item[(a)] Let $\omega\in\DD_k(\ZZ_p)$ such that $\pi(\omega)\in V_k'(\ZZ_p)$. Then we have $\omega\cdot_k\alpha_j^{-1}\in\DD_k(\ZZ_p)$ for $j\in\{0,\dots,p-1\}$.
\item[(b)]  Let $\omega\in F^n\DD_k(\ZZ_p)$. Then we have $\omega\cdot_k\alpha_j^{-1}\in F^{n+1}\DD_k(\ZZ_p)$ for $j\in\{0,\dots,p-1\}$.
\end{itemize}
\end{prop}

As pointed out before, in general the $\widetilde{U_p}$-operator does not respect the integral structures on our spaces of automorphic forms. However, if we put $\AA_k'(\Gamma,\ZZ_p)=\AA_k(\Gamma,V_k'(\ZZ_p))$ and $\mathcal{A}_k'(\Gamma,\ZZ_p)=\rho^{-1}(\AA_k'(\Gamma,\ZZ_p))\cap\mathcal{A}_k(\Gamma,\ZZ_p)$, Proposition \ref{helps} has the following consequences:
\pagebreak
\begin{cor}~
\label{kernel}
\begin{itemize}
\item[(a)] Let $\Phi\in\mathcal{A}_{k}'(\Gamma,\ZZ_p)$ such that $\widetilde{U_p}\rho(\Phi)=\rho(\Phi)$. Then $\widetilde{U_p}\Phi\in\mathcal{A}_{k}'(\Gamma,\ZZ_p)$.
\item[(b)] Let $\Phi\in\mathcal{A}_{k,n}(\Gamma,\ZZ_p)$. Then $\widetilde{U_p}\Phi\in\mathcal{A}_{k,n+1}(\Gamma,\ZZ_p)$.
\end{itemize}
\end{cor}

\begin{proof}
Property (b) is an immediate consequence of Proposition \ref{helps} (b). If $\Phi\in\mathcal{A}_{k}'(\Gamma,\ZZ_p)$, we have $\rho(\Phi)(g)=\pi(\Phi(g))\in V_k'(\ZZ_p)$ and we can apply part (a) of Proposition \ref{helps} to see that $\widetilde{U_p}\Phi\in\mathcal{A}_k(\Gamma,\ZZ_p)$. Since $\rho(\widetilde{U_p}\Phi)=\rho(\Phi)$, we conclude that $\widetilde{U_p}\Phi\in\mathcal{A}_{k}'(\Gamma,\ZZ_p)$.
\end{proof}

Now we can prove Theorem \ref{stevens}.

\begin{proof}[Proof of Theorem \ref{stevens}]
We first prove the injectivity. By Proposition \ref{multi}, it suffices to show that
\[
\mathcal{A}_k(\Gamma,\ZZ_p)^{U_p=p^{k/2}}\cap\mathrm{Ker}(\rho)=0.
\]
Let $\Phi\in\mathcal{A}_k(\Gamma,\ZZ_p)^{U_p=p^{k/2}}\cap\mathrm{Ker}(\rho)$. By definition, we have
\[
\mathcal{A}_k(\Gamma,\ZZ_p)\cap\mathrm{Ker}(\rho)=\mathcal{A}_{k,0}(\Gamma,\ZZ_p).
\]
Part (b) of the above proposition implies that $\Phi=(\widetilde{U_p})^n\Phi\in\mathcal{A}_{k,n}(\Gamma,\ZZ_p)$ for each $n\geq 0$, hence $\Phi=0$.\par
We now turn to the surjectivity. Let $\phi\in\AA_k(\Gamma)^{U_p=p^{k/2}}$. After multiplying with a suitably chosen scalar $u\in\QQ^\times_p$, we may assume $\phi\in\AA_k'(\Gamma,\ZZ_p)^{U_p=p^{k/2}}$.  Let $\widetilde{\Phi}\in\mathcal{A}_k'(\Gamma,\ZZ_p)$ be an arbitrary lift of $\phi$. By part (a) of the above proposition, we have
\[
\widetilde{\Phi}^n=(\widetilde{U_p})^n\widetilde{\Phi}\in\mathcal{A}_k'(\Gamma,\ZZ_p).
\]
We define $\Phi^n=\rho^n(\widetilde{\Phi}^n)\in\mathcal{A}_k^n(\Gamma,\ZZ_p)$. Then we have
\[
\widetilde{U_p}\widetilde{\Phi}-\widetilde{\Phi}\in\mathcal{A}_{k,0}(\Gamma,\ZZ_p)
\]
and thus by part (b) of the above proposition,
\[
\widetilde{\Phi}^{n+1}-\widetilde{\Phi}^n\in\mathcal{A}_{k,n}(\Gamma,\ZZ_p).
\]
This implies that $\rho^{n+1}_n(\Phi^{n+1})-\Phi^n=\rho^n(\widetilde{\Phi}^{n+1}-\widetilde{\Phi}^n)=0$ and since $\rho^n(\widetilde{\Phi}^{n+1})=\widetilde{U_p}\Phi^n$, we have constructed a compatible system of eigenforms $\Phi^n\in\mathcal{A}_k^n(\Gamma,\ZZ_p)^{U_p=p^{k/2}}$. By Proposition \ref{projlim}, the $\Phi^n$ glue together to an automorphic form
\[
\Phi\in\varprojlim_{n}\mathcal{A}_k^n(\Gamma,\ZZ_p)^{U_p=p^{k/2}}\simeq \mathcal{A}_k(\Gamma,\ZZ_p)^{U_p=p^{k/2}}.
\]
By construction, we have $\rho(\Phi)=\phi$, which completes the proof.
\end{proof}

\begin{rem}
\label{gencontr}
It is worth noting that Theorem \ref{stevens} and the construction described in the proof can easily be generalized to arbitrary $U_p$-eigenforms with eigenvalue $\lambda\in E_p$ such that $\nu_p(\lambda)<k+1$, where $E_p$ is a finite extension of $\Qp$. The proofs are completely analogous if one replaces $p^{\frac{k}{2}}$ by $\lambda$ and $\ZZ_p$ by the ring of integers $\mathcal{O}_{E_p}$ of $E_p$ and notes that the coefficient module $V_k'(\ZZ_p)$ needs to adapted accordingly. In the case of modular symbols this is done in \cite{grec}.
\end{rem}

\subsection{Automorphic forms and harmonic cocycles}

Following \cite[Section 5.2 and Section 6.1]{fm} we can identify the space of harmonic cocycles with a subspace of the space of $p$-adic automorphic forms and use the unique lift constructed in Theorem \ref{stevens} to compute the moments attached to a harmonic cocycle. 
\begin{defi}
Let $g_0=\left(\begin{smallmatrix} 0& 1\\ p& 0\end{smallmatrix}\right)\in\GL$. The \emph{Atkin-Lehner involution} $W_p$ on $\AA_k(\Gamma)$ is defined by
\[
(W_p\phi)(g)=\phi(gg_0)\cdot_k g_0^{-1}, \quad \text{for } g\in\GL, \phi\in\AA_k(\Gamma).
\]
Clearly, $W_p\phi\in\AA_k(\Gamma)$.
\end{defi}

Note that this definition does not extend to the space of rigid analytic automorphic forms since $g_0\not\in\QQ^\times_p\Sigma_0(p)$.

\begin{defi}
A $p$-adic automorphic form $\phi\in\AA_k(\Gamma)$ is said to be \emph{new at $p$} if
\[
U_p\phi=-p^{\frac{k}{2}}W_p\phi.
\]
The subspace of $p$-adic automorphic forms that are new at $p$ is denoted by $\AA_k(\Gamma)^{p\text{-new}}$.
\end{defi}

Let $\phi\in\AA_k(\Gamma)$. Let $c_\phi\colon\TT_1\rightarrow V_k$ be the map sending $e\in\TT_1$ to $c_\phi(e)=g\cdot_k \phi(g)$ where $e=ge_0$ for $g\in\GL$. This is well-defined since $\text{Stab}_\GL(e_0)=\QQ^\times_p\Gamma_0(p\ZZ_p)$ and $\QQ^\times_p$ acts trivially on $V_k$. Furthermore, the map $c_\phi$ satisfies
\[
c_\phi(\gamma e)=\gamma g\cdot_k\phi(\gamma g)=\gamma\cdot_k c_\phi(e), \quad \text{for all } e\in\TT_1, \gamma\in\Gamma.
\]
Let $e\in\TT_1$. By translating the actions of $U_p$ and $W_p$, we have $c_\phi\in C_h(\Gamma, V_k)$ for $\phi\in\left(\AA_k(\Gamma)^{\text{$p$-new}}\right)^{U_p=p^{k/2}}$. We obtain the following.

\begin{prop}
\label{cocaut}
The map
\begin{align*}
\left(\AA_k(\Gamma)^{\textnormal{$p$-new}}\right)^{U_p=p^{k/2}}&\rightarrow C_h(\Gamma, V_k)\\
\phi&\mapsto c_\phi
\end{align*}
is an isomorphism of $\QQ_p$-vector spaces.
\end{prop}

The following theorem tells us that the methods of the previous section can be applied to compute the moments of the measure associated to a harmonic cocycle $c\in C_h(\Gamma,V_k)$.

\begin{theo}
\label{speclift}
Let $c\in C_h(\Gamma,V_k)$ and $\phi_c\in\AA_k(\Gamma)$ be the corresponding $p$-adic automorphic form. Then the unique lift $\Phi_c$ of $\phi_c$ in $\mathcal{A}_k(\Gamma)^{U_p=p^{k/2}}$ in Theorem \ref{stevens} satisfies
\[
\Phi_c(g)(x^i)=m(\mu_c,g,i) \quad \text{for } g\in\GL, i\geq 0.
\]
\end{theo}

\begin{proof}
We define $\Phi_c\colon\GL\rightarrow \DD_k$ by
\[
\Phi_c(g)(f)=\int_{g\ZZ_p}(f\cdot_k g^{-1})(x)\text{d}\mu_c(x)\in\QQ_p
\]
for $g\in\GL$ and $f\in \Ta_k$. Note that we may view $f$ as an element of $A_k^\mathrm{loc}$ by extending with $0$ outside $\ZZ_p$ and thus the above integral is well-defined. Note that the inclusion $\Ta_k \rightarrow A_k^\mathrm{loc}$ is compatible with the $\Gamma_0(p\ZZ_p)$-actions and continuous. By the continuity of $\mu_c$, we have $\Phi_c(g)\in\DD_k$ for all $g\in\GL$. Let $u\in\QQ^\times_p$,$\gamma\in\Gamma$ and $\sigma\in\Gamma_0(p\ZZ_p)$. Then we get
\begin{align*}
\Phi_c(u\gamma g \sigma)(f)&=\int_{u\gamma g\sigma\ZZ_p}(f\cdot_k(u\gamma g\sigma)^{-1})(x)\text{d}\mu_c(x)\\
&=\int_{g\ZZ_p} ((\sigma\cdot_k f)\cdot_k g^{-1})(x)\text{d}\mu_c(x)\\
&=\Phi_c(g)(\sigma\cdot_k f),
\end{align*}
and thus $\Phi_c\in\mathcal{A}_k(\Gamma)$.  It remains to show that $\rho(\Phi_c)=\phi_c$ and $U_p\Phi_c=p^{\frac{k}{2}}\Phi_c$. For $i\in\{0,\dots,k\}$ we have
\begin{align*}
\Phi_c(g)(x^i)&=\int_{g\ZZ_p} x^i\cdot_k g^{-1}\text{d}\mu_c(x)=\mu_c(g\ZZ_p)(x^i\cdot_k g^{-1})\\
&=(g^{-1}\cdot_k c(g e_0))(x^i)=\phi_c(g)(x^i).
\end{align*}
Since $\bigsqcup_{j=0}^{p-1}\alpha_j \ZZ_p=\ZZ_p$, we compute
\begin{align*}
(U_p\Phi_c)(g)(f)&=p^{\frac{k}{2}}\sum_{j=0}^{p-1}\Phi_c(g\alpha_j)(f\cdot_k \alpha_j)=p^{\frac{k}{2}}\sum_{j=0}^{p-1}\int_{g\alpha_j\ZZ_p}(f\cdot_k g^{-1})(x)\text{d}\mu_c(x)\\
&=p^{\frac{k}{2}}\int_{g\ZZ_p}(f\cdot_k g^{-1})(x)\text{d}\mu_c(x)=p^{\frac{k}{2}}\Phi_c(g)(f).
\end{align*}
The uniqueness of the lift in Theorem \ref{stevens} completes the proof.
\end{proof}

Let $c\in C_h(\Gamma,V_k)$. We may rescale such that $\phi_c\in\AA_k'(\Gamma,\ZZ_p)$ and $\Phi_c\in\mathcal{A}_k'(\Gamma,\ZZ_p)$. Let $\Phi_0\in\mathcal{A}_k'(\Gamma,\ZZ_p)$ be an arbitrary lift. Then we can rephrase the construction of $\Phi_c$ given in the proof of Theorem \ref{stevens} as follows.

\begin{cor}
For all $n\geq 1$ we have
\[
((\widetilde{U_p})^n\Phi_0)(g)(x^{k+i})=m(\mu_c,g,k+i) \quad (\textnormal{mod } p^{n-i+1})
\]
for $i\in\{1,\dots, n\}$ and $g\in\GL$. 
\end{cor}

In other words, we just need to compute an arbitrary initial lift $\Phi_0\in\mathcal{A}_k'(\Gamma,\ZZ_p)$ of $\phi_c$ and repeatedly apply the operator $\widetilde{U_p}$ in order to compute the moments to the desired precision. We have explicitly constructed such an initial lift in Proposition \ref{initiallift}, which can be readily computed, assuming that the values of $\phi_c$ are known. Of course one has to take the scaling factor arising from the possible denominators in in the initial lift into account when deciding how many times the $\widetilde{U_p}$-operator has to be applied in order to compute the moments to the desired precision.

\begin{rem}
By carefully examining the explicit formula for the $\widetilde{U_p}$-operator, we see that one could lose precision when computing the values $(\widetilde{U_p}\Phi_0)(g)(x^i)$ for $i\leq k/2$. But since
\[
(\widetilde{U_p}\Phi_0)(g)(x^i)=\rho(\widetilde{U_p}\Phi_0)(g)(x^i)=\phi_c(x^i) \quad \textnormal{for } i\leq k,
\]
one can replace the low moments by the corresponding values of $\phi_c$, which are assumed to be known to the desired precision.
\end{rem}

\section{Computation of the $\mathcal{L}$-operator}
\label{compu}
In this section, we describe our method to compute the $\mathcal{L}$-operator building on \cite{fm}.
\subsection{Computing fundamental domains}
The following lemma is essential for the algorithm  to compute fundamental domains for the action of $\Gamma$ on $\mathcal{T}$ in \cite{fm}.
\begin{lemm}[{\cite[Lemma 2.2]{fm}}]
\label{normalization}
There is a set of coset representatives $\{g_i\}_{i\in I}$ for $\GL/\QQ^\times_p\Gamma_0(p\ZZ_p)\simeq\mathcal{T}_1$ given by matrices with coefficients in $\ZZ$. Moreover, there is an effective algorithm that finds for a given matrix $g\in\GL$ a scalar $u\in\QQ^\times_p$ and a matrix $\sigma\in\Gamma_0(p\ZZ_p)$ such that $gu\sigma=g_i$ for some $i\in I$. An analogous statement holds for $\GL/\QQ^\times_p\textnormal{GL}_2(\ZZ_p)\simeq\mathcal{T}_0$.
\end{lemm}

The algorithm used in the proof of the above lemma is used to encode vertices and edges of $\TT$ in terms of fixed matrix representatives. For a given matrix in $\GL$ we refer to the corresponding representative as the \emph{normalization} of this matrix. Another feature of the normalization is that that if $g_i$ is the normalization of a matrix $g\in\GL$, we have $d(v_0,g v_0)=\nu_p(\det(g_i))$, where $\nu_p$ is the $p$-adic valuation, normalized such that $\nu_p(p)=1$. The main result of \cite{fm} is stated below.

\begin{theo}[{\cite[Theorem 1.1]{fm}}]
There exists an effective algorithm for computing the following data:
\begin{itemize}
\item[(1)] A finite connected subtree $\mathcal{F}$ of $\mathcal{T}$, whose edges comprise a complete set of distinct orbit
representatives for the action of $\Gamma$ on the edges of $\mathcal{T}$.
\item[(2)] The stabilizers $\mathrm{Stab}_\Gamma(v)$, $\mathrm{Stab}_\Gamma(e)$ for all $v\in\mathcal{F}_0$, $e\in\mathcal{F}_1$.
\item[(3)]  A pairing of the boundary vertices of $\mathcal{F}$ that describes how the boundary vertices are identified
in the quotient graph $\Gamma\textbackslash\mathcal{T}$.
\end{itemize}
\end{theo}

For a detailed description and analysis of the complexity of the above algorithm see \cite[Section 3]{fm}. The crucial step is to reduce the problem wether two edges (or vertices) are $\Gamma$-equivalent to a shortest vector search in a lattice by only using an approximation of the fixed splitting up to some finite $p$-adic precision.\par Starting from this point, Franc and Masdeu can also easily compute a basis for $C_h(\Gamma, V_k)$ and work with the spaces of $p$-adic automorphic forms by Remark \ref{cosrep}, which makes the methods of the previous section applicable to our setting.

\subsection{From general integrals to moments}
By the results of Section \ref{autf}, in order to compute $\lambda^\tau$ efficiently, it remains to show that the relevant integrand admits a nice analytic expression on a certain covering of $\PP^1(\Qp)$ by disjoint compact open balls. We will construct an explicit covering $\mathcal{U}_{\tau_1,\tau_2}$ attached to $\tau_1,\tau_2\in\Hp(K_p)$ such that, for arbitrary $P\in \mathcal{P}_k$, the function
$P(x)\log_p\left(\frac{x-\tau_2}{x-\tau_1}\right)$ admits such an analytic expression on each ball $U\in\mathcal{U}_{\tau_1,\tau_2}$. This makes it possible to compute the corresponding line integral and hence in particular the map $\lambda^\tau$.  The idea to use the special covering we describe in this section is implicitly in the code by Franc and Masdeu built on \cite{fm} and we give all the details. This is an extension of the method described in \cite{dp}.\par
Since $K_p$ is an unramified extension of $\Qp$, the points $\tau_1,\tau_2\in\mathcal{H}_p(K_p)$ are mapped to vertices of $\TT$ under the reduction map. Let $n=d(\operatorname{\mathrm{red}}(\tau_1),\operatorname{\mathrm{red}}(\tau_2))$ and set $w_0=\operatorname{\mathrm{red}}(\tau_2)$ and $w_n=\operatorname{\mathrm{red}}(\tau_1)$. Let $w_1,\dots, w_{n-1}$ be the vertices on the geodesic $w_0\rightarrow w_n$ between $w_0$ and $w_n$ labeled such that $w_{j+1}$ is adjacent to $w_{j}$ for $j\in\{0,\dots,n-1\}$. Let
\[
\TT(w_0,w_n)=\{e\in\TT_1 \mid s(e)=w_j \text{ for a } j\in\{0,\dots,n\}, t(e)\neq  w_j \text{ for all } j\in\{0,\dots, n\}\}
\]
denote the finite set of all edges leaving the geodesic. By definition, the collection of  disjoint compact open balls
\[
\mathcal{U}_{\tau_1,\tau_2}=\{U_e\in\mathcal{B} \mid e\in\TT(w_0,w_n)\}\subset\mathcal{B}
\]
is a covering of $\PP^1(\Qp)$. Let $j\in\{0,\dots ,n\}$. We denote by $g_{j,\nu}\in\GL$, $\nu\in\{0, \dots,n_j\}$ the normalized matrix representatives for the edges in $\TT(w_0,w_n)\cap\TT(w_j)$ as in Lemma \ref{normalization}. Note that  for $n>0$ we have $n_j=p-1$ for $j\in\{0,n\}$ and $n_j=p-2$ otherwise. In the case $n=0$ we have $n_0=p$. We get $U_{g_{j,\nu}e_0}=g_{j,\nu}\ZZ_p$. Thus, we may write
\begin{align*}
\int_{\tau_1}^{\tau_2}\omega_c(P)&=\sum_{j=0}^n\sum_{\nu=0}^{n_j}\int_{g_{j,\nu}\ZZ_p}P(x)\log_p\left(\frac{x-\tau_2}{x-\tau_1}\right)\text{d}\mu_c(x)\\
&=\sum_{j=0}^n\sum_{\nu=0}^{n_j}\int_{g_{j,\nu}\ZZ_p}f_{j,\nu}\cdot_k g_{j,\nu}^{-1}(x)\text{d}\mu_c(x),
\end{align*}
where
\[
f_{j,\nu}(x)=\left(P(x)\log_p\left(\frac{x-\tau_2}{x-\tau_1}\right)\right)\cdot_k g_{j,\nu}=(P\cdot_k g_{j,\nu})(x)l_{j,\nu}(x)
\]
and
\[
l_{j,\nu}(x)=\log_p\left(\frac{g_{j,\nu}x-\tau_2}{g_{j,\nu}x-\tau_1}\right).
\]
For simplicity, we assume from now on that $P$ is in fact a monomial in $\{1,x,\dots,x^k\}$. 
\begin{prop}
\label{anaexpr}
For all $j\in\{0,\dots,n\}$ and $\nu\in\{0,\dots,n_j\}$ the function $f_{j,\nu}$ is analytic on $\ZZ_p$ and has a series expansion of the form
\[
f_{j,\nu}(x)=a_{0,j,\nu}p^{-\frac{k}{2}\nu_p(\det(g_{j,\nu}))}+\sum_{i=1}^\infty a_{i,j,\nu}p^{r_{i,j,\nu}}x^i
\]
with
\[
r_{i,j,\nu}=\max\{i-k,0\}-\frac{k}{2}\nu_p(\det(g_{j,\nu}))-\left\lfloor\frac{\log(i)}{\log(p)}\right\rfloor
\]
and $a_{i,j,\nu}$ belonging to $\mathcal{O}_p$, the ring of integers of $K_p$, for all $i\in\NN_0$.
\end{prop}

Note that the continuity of $\mu_c$ then implies that
\[
\int_{\tau_1}^{\tau_2}\omega_c(P)=\sum_{j=0}^n\sum_{\nu=0}^{n_j}\left(a_{0,j,\nu}p^{-\frac{k}{2}\nu_p(\det(g_{j,\nu}))}m(\mu_c,g_{j,\nu},0)+\sum_{i=1}^\infty a_{i,j,\nu}p^{r_{i,j,\nu}}m(\mu_c,g_{j,\nu},i)\right).
\]
Together with the methods of the previous section, this gives an efficient way to compute the above integral up to a prescribed precision.
The crucial step to prove Proposition \ref{anaexpr} is to find an analytic expression for the function $l_{j,\nu}$.

\begin{prop}
For $j\in\{0,\dots,n\}$ and $\nu\in\{0,\dots,n_j\}$ the function $l_{j,\nu}$ is analytic on $\ZZ_p$ and has a series expansion of the form
\[
l_{j,\nu}(x)=a_{0,j,\nu}+\sum_{i=1}^\infty a_{i,j,\nu}\frac{p^i}{i}x^i
\]
with $a_{i,j,\nu}\in\mathcal{O}_p$ for all $i\in\NN_0$.
\end{prop}

\begin{proof}
We have
\begin{align*}
l_{j,\nu}(x)=\log_p\left(\frac{g_{j,\nu}x-\tau_2}{g_{j,\nu}x-\tau_1}\right)=\log_p\left(\frac{g_{j,\nu}x-g_{j,\nu}g_{j,\nu}^{-1}\tau_2}{g_{j,\nu}x-g_{j\nu}g_{j,\nu}^{-1}\tau_1}\right)=\log_p\left(\frac{x-g_{j,\nu}^{-1}\tau_2}{x-g_{j,\nu}^{-1}\tau_1}\right)+C,
\end{align*}
where 
\[
C=\log_p\left(\frac{cg_{j,\nu}^{-1}\tau_1+d}{cg_{j,\nu}^{-1}\tau_2+d}\right)\in\mathcal{O}_p, \quad g_{j,\nu}=\begin{pmatrix} a&b\\ c&d\end{pmatrix}\in\GL.
\]
Consequently, we obtain
\[
l_{j,\nu}(x)=C+\log_p\left(\frac{g_{j,\nu}^{-1}\tau_2}{g_{j,\nu}^{-1}\tau_1}\right)+\log_p\left(\frac{1-\frac{x}{g_{j,\nu}^{-1}\tau_2}}{1-\frac{x}{g_{j,\nu}^{-1}\tau_1}}\right).
\]
Let $e\in\TT_1$ be an arbitrary edge. We denote by $\TT_e^{c}$ the maximal subtree of $\TT$ containing $s(e)$ and not containing $e$. Note that $\TT$ is the union of the tree $\TT_e$ defined in Subsection \ref{btt} and $\TT_e^c$ and these trees intersect precisely in $s(e)$. By definition of $\TT(w_0,w_n)$, we have $w_0,w_n\in(\TT_{g_{j,\nu}e_0}^c)_0$. Thus, we obtain
\[
g_{j,\nu}^{-1}w_0, g_{j,\nu}^{-1}w_n\in(\TT_{e_0}^c)_0.
\]
By definition of the reduction map, this implies $|g_{j,\nu}^{-1}\tau_1|_p\geq p$ and $|g_{j,\nu}^{-1}\tau_2|_p\geq p$. Since $x\in\ZZ_p$, we can use the geometric series and the series defining the logarithm,
\[
\log_p(1+x)=\sum_{i=1}^\infty\frac{(-1)^{i+1}}{i}x^i, \quad x\in p\mathcal{O}_p,
\]
to get the desired series expansion.
\end{proof}

\begin{proof}[Proof of Proposition \ref{anaexpr}]
By the above proposition, we have
\[
l_{j,\nu}(x)=a_{0,j,\nu}+\sum_{i=1}^\infty a_{i,j,\nu}\frac{p^i}{i}x^i
\]
with $a_{i,j,\nu}\in\mathcal{O}_p$ for all $i\in\NN_0$. We also have
\[
P\cdot_k g_{j,\nu}(x)=p^{-\frac{k}{2}\nu_p(\det(g_{j,\nu}))}\cdot\sum_{i=0}^k b_{i,j,\nu} x^i
\]
with $b_{i,j,\nu}\in\ZZ_p$ for all $i\in\{0,\dots, k\}$. By multiplying, we get a series expansion of the desired form.
\end{proof}

An examination of the formula in Proposition \ref{anaexpr} shows that the integral may be computed up to a finite $p$-adic precision $M$ by computing for each $j\in\{0,\dots, n\}$ and $\nu\in\{0,\dots,n_j\}$ the following data:
\[
m(\mu_c,g_{j,\nu},i) \quad (\textnormal{mod } p^{M-r_{i,j,\nu}}) ,\quad  i\in\{0,\dots,N_{j,\nu}\}, 
\]
with $N_{j,\nu}=\max\{i\in\NN_0 \mid r_{i,j,\nu}< M\}$ and $r_{0,j,\nu}=-\frac{k}{2}\nu_p(\det(g_{j,\nu}))$. Together with the methods described in Section \ref{autf}, this describes an efficient way to compute the map $\lambda^\tau$. 

\begin{rem}
It is worth noting that this method can be used to integrate other locally analytic functions against $\mu_c$ as long as one can compute a locally analytic decomposition for the integrand. See for example \cite[Section 6.1]{fm}, where a similar method is used to compute values of rigid analytic modular forms using the Poisson kernel described in Subsection \ref{poi}.
\end{rem}

\section{Examples}
\label{exa}

After giving some elementary examples, we primarily focus on phenomena which occur in the distribution of slopes of $\mathcal{L}$-invariants for $p=2$. Throughout this section, we assume $N^{+}=1$ and consequently, $N=N^{-}$.
\subsection{An elementary example}

In our first example, we put $p=3$, $N^{-}=2$. Let $\eta$ denote the Dedekind $\eta$-function and set
\[
f(\tau)=(\eta(\tau)\eta(2\tau)\eta(3\tau)\eta(6\tau))^2 \in\mathcal{S}_4(\Gamma_0(6))^{\mathrm{new}}.
\]
The correponding harmonic cocycle $c_f$ spans $C_h(\Gamma,V_2)$. Thus, we can apply our algorithm to obtain
\[
\mathcal{L}_3(f)=1 +  3^{2} + 2 \cdot 3^{7} + 3^{8} + 2 \cdot 3^{9} \quad (\text{mod } 3^{10}).
\]
This example was also computed by Teitelbaum in \cite[Section 3]{tei90} up to $3$-adic precision $5$. If we increase the weight to $6$, the $\mathcal{L}$-operator is again given by multiplication with a scalar. We compute
\[
\mathcal{L}=(3^{-1} +  2 \cdot 3^{2} + 3^{4} + 3^{7} + 2 \cdot 3^{8} + 2 \cdot 3^{9} )\cdot\operatorname{\mathrm{id}}_{H^1(\Gamma,V_4)} \quad (\text{mod } 3^{10})
\]
and note that this an example of a non-integral $\mathcal{L}$-invariant.
\subsection{Slopes of $\mathcal{L}$-invariants for $p=2$}
In the remainder, we want to analyze the distribution of (the slopes of) $\mathcal{L}$-invariants for growing weight.\par For $k\in\ZZ$ even,  $k\geq 2$ we denote by $\alpha_\mathcal{L}(k,p,N)$ the (finite) sequence of slopes of the $p$-adic $\mathcal{L}$-operator of weight $k$ acting on $H^1(\Gamma, V_{k-2})$ and set $d_k(p,N)=\dim_\Qp H^1(\Gamma, V_{k-2})$. Moreover, we denote by $\alpha_\mathcal{L}^{+}(k,p,N)$ the (finite) sequence of slopes of the $p$-adic $\mathcal{L}$-operator of weight $k$ acting on $H^1(\Gamma, V_{k-2})^{W_N=1}$ and similarly by $\alpha_\mathcal{L}^{-}(k,p,N)$ the (finite) sequence of slopes of the $p$-adic $\mathcal{L}$-operator of weight $k$ acting on $H^1(\Gamma, V_{k-2})^{W_N=-1}$. Here, $W_N$ denotes the Atkin-Lehner involution at $N$. Finally, let $\varepsilon_W(k,p,N)$ denote the (finite) sequence of eigenvalues of the Atkin-Lehner involution $W_p$ at $p$ acting on this space.

\begin{rem}
It should be pointed out that we consider the Atkin-Lehner involutions arising from the space $\mathcal{S}_k(\Gamma_0(pN))^{\mathrm{new}}$. This is important to note since the Jacquet-Langlands correspondence together with the $p$-adic unformization interchange the $\pm 1$ eigenspaces for these operators. More details can be found in \cite[Section 5]{kli}.
\end{rem}

We first consider the case $N=3$.

\begin{center}
\begin{table}[h]
\begin{tabular}{|p{0.7 cm}|p{2 cm}||>{\centering\arraybackslash}p{4cm}|>{\centering\arraybackslash}p{4cm}||>{\centering\arraybackslash}p{2cm}|}
\hline
$k$ & $d_k(2,3)$ & \multicolumn{2}{c||}{$\alpha_\mathcal{L}(k,2,3)$} & $\epsilon_W(k,2,3)$\\
& & \multicolumn{1}{c}{$\alpha_\mathcal{L}^+(k,2,3)$} & \multicolumn{1}{c||}{$\alpha_\mathcal{L}^-(k,2,3)$} & \\
\hline
$4$     &$1$& $1_1$ & & $1_1$\\
\hline
$6$     &$1$& $0_1$ & &  $-1_1$\\
\hline
$8$     &$1$&  & $-1_1$ & $-1_1$\\ 
\hline
$10$     &$1$&  & $0_1$ & $1_1$\\ 
\hline
$12$     &$3$& $-1_1$ & $-4_2$ & $-1_1, 1_2$\\ 
\hline
$14$     &$1$& $-1_1$ &  & $-1_1$\\ 
\hline
$16$     &$3$& $-4_2$ & $-2_1$ & $-1_2,1_1$\\ 
\hline
$18$     &$3$& $-4_2$ & $-1_1$ & $-1_1, 1_2$\\ 
\hline
$20$     &$3$& $-2_1$ & $-6_2$ & $-1_1, 1_2$\\ 
\hline
$22$     &$3$& $-2_1$ & $-4_2$ & $-1_2, 1_1$\\ 
\hline
$24$     &$5$& $-6_2$ & $-2_1,-7_2$& $-1_3, 1_2$\\ 
\hline
$26$     &$3$& $-6_2$ & $-2_1$ & $-1_1, 1_2$\\ 
\hline
$28$     &$5$& $-2_1,-7_2$ & $-5_2$ & $-1_2,1_3$\\ 
\hline
$30$     &$5$& $-2_1,-7_2$ & $-6_2$& $-1_3,1_2$\\ 
\hline
$32$     &$5$& $-5_2$ & $-3_1,-7_2$& $-1_3,1_2$\\ 
\hline
$34$     &$5$& $-5_2$ & $-2_1, -7_2$& $-1_2,1_3$\\ 
\hline
$36$     &$7$& $-3_1,-7_2$ & $-6_2, -10_2$ & $-1_3, 1_4$\\ 
\hline
$38$     &$5$& $-3_1,-7_2$ & $-5_2$& $-1_3,1_2$\\ 
\hline
$40$     &$7$& $-6_2,-10_2$ & $-3_1, -9_2$& $-1_4,1_3$\\ 
\hline
\end{tabular}
\bigskip
\caption{$p=2$, $N=3$}
\end{table}
\end{center}

The computational results suggest various links between slopes of different weights, which are made precise in the following. 
\begin{con}
\label{con1}
For $k\in 4\ZZ$, $k\geq 8$, we have
\[
\alpha_\mathcal{L}^-(k,2,3)=\alpha_\mathcal{L}^+(k+4,2,3)=\alpha_\mathcal{L}^+(k+6,2,3)=\alpha_\mathcal{L}^-(k+10,2,3)
\]
and consequently
\[
\alpha_\mathcal{L}(k,2,3)=\alpha_\mathcal{L}(k+6,2,3) \quad \text{as well as} \quad
\epsilon_W(k,2,3)=\epsilon_W(k+6,2,3).
\]
\end{con}
The above table verifies this statement up to $k=28$ computationally.
\begin{rem}
By the explicit dimension formulas in \cite{ma}, it is easy to verify that 
\[
d_k(2,3)=\dim_\CC \mathcal{S}_{k}(\Gamma_0(6))^{\mathrm{new}}= \dim_\CC \mathcal{S}_{k+6}(\Gamma_0(6))^{\mathrm{new}}=d_{k+6}(2,3)
\]
for $k$ as above. It should also be noted that the statement on the $\mathcal{L}$-invariants concerns only the slopes, the $\mathcal{L}$-invariants themselves are not equal, which is shown by our computations. Interestingly, while the Atkin-Lehner involution at $3$ makes the relations between $\mathcal{L}$-invariants visible, as stated above, the Atkin-Lehner eigenvalues at $2$ follow the same pattern.
\end{rem}
Another interesting observation is that aside from the highest slope, which appears with multiplicity one, all other slopes appear in pairs.\par
\vspace{\baselineskip}
Next, we consider the case $N=5$.

\begin{center}
\begin{table}[h]

\begin{tabular}{|p{0.7 cm}|p{2 cm}||>{\centering\arraybackslash}p{4 cm}|>{\centering\arraybackslash}p{4 cm}||>{\centering\arraybackslash}p{2cm}|}
\hline
$k$ & $ d_k(2,5)$ & \multicolumn{2}{c||}{$\alpha_\mathcal{L}(k,2,5)$} & $\epsilon_W(k,2,5)$\\
& & \multicolumn{1}{c}{$\alpha_\mathcal{L}^+(k,2,5)$} & \multicolumn{1}{c||}{$\alpha_\mathcal{L}^-(k,2,5)$} & \\
\hline
$4$     &$1$&  &$2_1$ & $-1_1$\\
\hline
$6$     &$3$& $-2_2$ & $0_1$ & $-1_1, 1_2$\\
\hline
$8$     &$1$&  &$-1_1$ & $-1_1$\\ 
\hline
$10$     &$3$& $-5_2$ & $0_1$ &$-1_1,1_2$\\ 
\hline
$12$     &$5$& $-2_2$ & $-1_1, -4_2$ & $-1_3, 1_2$\\ 
\hline
$14$     &$3$& $-3_2$ &$-1_1$  & $-1_1, 1_2$ \\ 
\hline
$16$     &$5$& $-5_2$ &$-2_1,-4_2$ & $-1_3, 1_2$\\ 
\hline
$18$     &$7$& $-5_4$ &$-1_1,-4_2$ & $-1_3, 1_4$\\ 
\hline
$20$     &$5$& $-3_2$ & $-2_1,-6_2$ & $-1_3, 1_2$\\ 
\hline
$22$     &$7$& $-1_2,-8_2$ &$-2_1,-4_2$ & $-1_3, 1_4$\\ 
\hline
$24$     &$9$& $-5_4$ & $-2_1,-6_2,-7_2$ & $-1_5, 1_4$\\ 
\hline
$26$     &$7$& $-\tfrac{11}{2}_4$ &$-2_1, -6_2$ & $-1_3, 1_4$\\ 
\hline
$28$     &$9$& $-1_2, -8_2$ & $-2_1,-5_2,-7_2$ & $-1_5, 1_4$\\ 
\hline
$30$     &$11$& $-4_2, -7_4$ &$-2_1,-6_2,-7_2$ & $-1_5, 1_6$\\ 
\hline
$32$     &$9$ & $-\tfrac{11}{2}_4$ & $-3_1,-5_2,-7_2$ & $-1_5, 1_4$\\ 
\hline
$34$     &$11$& $-\tfrac{11}{2}_4,-11_2$ &$-2_1,-5_2,-7_2$ & $-1_5, 1_6$\\ 
\hline
$36$     &$13$& $-4_2, -7_4$ & $-3_1,-6_2,-7_2,-10_2$ & $-1_7,1_6$ \\ 
\hline
$38$     &$11$& $-3_2, -8_4$ & $-3_1,-5_2,-7_2$  & $-1_5, 1_6$ \\ 
\hline
$40$     &$13$& $-\tfrac{11}{2}_4, -11_2$ & $-3_1,-6_2,-9_2,-10_2$ & $-1_7, 1_6$\\ 
\hline
\end{tabular}
\bigskip
\caption{$p$=2, $N=5$}
\end{table}
\end{center}

A priori, the strong statement of Conjecture \ref{con1} does not hold for $N=5$, but we still see relations between the slopes for $k$ and $k+6$ in the third column. Even more surprisingly,  the slopes for $N=3$ are exactly the ones appearing in the fourth column above, i.e. in the subspace defined by $W_5=-1$. 
Moreover, in the fifth column, we note that the Atkin-Lehner eigenvelues at $2$ from  the case $N=3$ appear as well. To formulate our observations more precisely, we let $\epsilon_W^\ast(k,2,5)$ denote the values in $\epsilon_W(k,2,5)$ not ``\emph{coming from}'' the case $N=3$.

\begin{con}
\label{con3}
For $k\geq 6$ we have
\[
\alpha_\mathcal{L}^-(k,2,5)= \alpha_\mathcal{L}(k,2,3).
\]
Moreover, for $k\in 2+4\ZZ$, $k\geq 6$, we have
\[
\alpha_\mathcal{L}^+(k,2,5)=\alpha_\mathcal{L}^+(k+6,2,5) \quad \text{and} \quad
\epsilon_W^\ast(k,2,5)=-\epsilon_W^\ast(k+6,2,5).
\]
\end{con}
In particular, the first part of the above conjecture suggests very unexpected relations between newforms for different levels, while combining the second part with Conjecture \ref{con1} gives very explicit relations between slopes of different weights.
\begin{rem}
The relations between the eigenvalues of $W_2$ are now given by multiplication with $-1=(-1)^{k/2}$ which could be related to the fact that here we have $k\in 2+4\ZZ$. Note however, that $\epsilon_W^\ast(k,2,5)$ are \emph{not} the eigenvalues on the subspace defined by $W_5=1$ as one could naturally expect, but we merely removed the values appearing for $N=3$ (with multiplicity). The eigenvalues of $W_2$ on the subspace defined by $W_5=1$ are in fact always evenly distributed between $1$ and $-1$.
\end{rem}
Note that the statement on pairs of slopes for the first case is still valid, even though for $k=38$ we see that the slopes of one pair coincides with the slope of multiplicity one and for $k=22$ the multiplicity one slope is not the highest slope.\par
\vspace{\baselineskip}
Finally, we consider the case $N=7$.

\begin{center}
\begin{table}[h]

\begin{tabular}{|p{0.7 cm}|p{2 cm}||>{\centering\arraybackslash}p{4cm}|>{\centering\arraybackslash}p{4cm}||>{\centering\arraybackslash}p{1.8cm}|}
\hline
$k$ & $ d_k(2,7)$ & \multicolumn{2}{c||}{$\alpha_\mathcal{L}(k,2,7)$} & $\epsilon_W(k,2,7)$\\
& & \multicolumn{1}{c}{$\alpha_\mathcal{L}^+(k,2,7)$} & \multicolumn{1}{c||}{$\alpha_\mathcal{L}^-(k,2,7)$} & \\
\hline
$4$     &$2$& $1_1$ &$1_1$ & $-1_1,1_1$\\
\hline
$6$     &$2$& $0_1$ & $0_1$ & $-1_1, 1_1$\\
\hline
$8$     &$4$& $0_1,-1_1$ &$0_1,-1_1$ & $-1_3,1_1$\\ 
\hline
$10$     &$4$& $0_1,-1_1$ & $0_1,-1_1$ &$-1_1,1_3$\\ 
\hline
$12$     &$6$& $-1_1,-4_2$ & $-1_1, -4_2$ & $-1_3, 1_3$\\ 
\hline
$14$     &$6$& $-1_1,-4_2$ &$-1_1, -4_2$  & $-1_3, 1_3$ \\ 
\hline
$16$     &$8$& $-1_1, -2_1,-4_2$ &$-1_1,-2_1,-4_2$ & $-1_5, 1_3$\\ 
\hline
$18$     &$8$& $-1_1,-2_1,-4_2$ &$-1_1,-2_1,-4_2$ & $-1_3, 1_5$\\ 
\hline
$20$     &$10$& $-2_1,-4_2,-6_2$ & $-2_1,-4_2,-6_2$ & $-1_5, 1_5$\\ 
\hline
$22$     &$10$& $-2_1,-4_2,-6_2$ &$-2_1,-4_2, -6_2$ & $-1_5, 1_5$\\ 
\hline
$24$     &$12$& $-2_2,-6_2,-7_2$ & $-2_2,-6_2,-7_2$ & $-1_7, 1_5$\\ 
\hline
$26$     &$12$&  $-2_2,-6_2,-7_2$ &$-2_2, -6_2, -7_2$ & $-1_5, 1_7$\\ 
\hline
$28$     &$14$& $-2_1, -5_2, -6_2, -7_2$ & $-2_1,-5_2,-6_2,-7_2$ & $-1_7, 1_7$\\ 
\hline
$30$     &$14$& $-2_1, -5_2, -6_2, -7_2$ &$-2_1,-5_2,-6_2,-7_2$ & $-1_7, 1_7$\\ 
\hline
$32$     &$16$ & $-2_1,-3_1,-5_2,-7_4$ & $-2_1,-3_1,-5_2,-7_4$ & $-1_9, 1_7$\\ 
\hline
$34$     &$16$& $-2_1,-3_1,-5_2,-7_4$ & $-2_1,-3_1,-5_2,-7_4$ & $-1_7, 1_9$\\ 
\hline
$36$     &$18$& $-3_1, -5_2, -6_2, -7_2, -10_2$ & $-3_1,-5_2,-6_2,-7_2,-10_2$ & $-1_9,1_9$ \\ 
\hline
$38$     &$18$& $-3_1, -5_2, -6_2, -7_2, -10_2$& $-3_1, -5_2, -6_2, -7_2, -10_2$  & $-1_9, 1_9$ \\ 
\hline
$40$     &$20$& $-3_2, -6_2, -7_2, -9_2, -10_2$ & $-3_2, -6_2, -7_2, -9_2, -10_2$  & $-1_{11}, 1_9$ \\ 
\hline
\end{tabular}
\bigskip
\caption{$p$=2, $N=7$}
\end{table}
\end{center}

Even more clearly as above, we observe relations between slopes of different weights, in this case between successive weights. The Atkin-Lehner eigenvalues in the fifth column seem to behave similarly to the case $N=5$. Moreover, the third and fourth column are completely identical, which is rather different from the previous cases. However, the slopes appearing  in for $N=3$ are again appearing in the the fourth (and third) column. We collect our observations in the following.

\begin{con}~
\begin{itemize}
\item[(i)] We have $\alpha_\mathcal{L}^+(k,2,7)=\alpha_\mathcal{L}^-(k,2,7)$ for all $k\geq 4$.
\item[(ii)] The slopes apearing in $\alpha_\mathcal{L}(k,2,3)$ appear in $\alpha_\mathcal{L}^-(k,2,7)$ as well.
\item[(ii)] For $k\in 4\ZZ$, $k\geq 8$ we have
\[
\alpha_\mathcal{L}(k,2,7)=\alpha_\mathcal{L}(k+2,2,7) \quad \text{and} \quad
\epsilon_W(k,2,7)=-\epsilon_W(k+2,2,7).
\]
\end{itemize}
\end{con}
Again, we note that in this case \emph{all} slopes appear in pairs and are evenly distributed in the $\pm 1$ eigenspaces for $W_7$.

\begin{rem}
As above, by the explicit dimension formulas in \cite{ma} we can directly verify that  $d_k(2,7)=d_{k+2}(2,7)$ for $k$ as above.
\end{rem}

At present, we do not have any explanation for the phenomena described in the above conjectures, but our computations give convincing evidence. In light of the Greenberg-Stevens formula $\mathcal{L}_p(f)=-2\operatorname{\mathrm{dlog}}(a_p(\kappa))|_{\kappa=k}$ and similar observations for the slopes of the $U_p$-operator in \cite[Consequence 4.6]{buz} or \cite[Example 3]{vonk} that suggest relations between different eigencurves, it would be very interesting to compute the quantity  $a_p'(k)$ also at the boundary of the weight space. However, our method is naturally restricted to classical points. Another approach we are currently pursuing building on the methods of \cite{lau} and \cite{vonk} could address this question. First results for $N=1$ in that direction are promising. 

Finally, we should mention that the computations for $p=3$ have not revealed strong evidence for similar relations yet. However, due to the long running time of these computations the amount of data in this case is still very small ($k\leq 20$). Since the increment and the threshold after which the patterns are visible could possibly change with the prime $p$, we are not in a position to make a decisive statement on whether these phenomena are limited to the case $p=2$. We are in the process of gathering data for many other cases including $p=5$ as well.

\end{document}